\documentclass[oneside,english,a4paper]{amsart}
\usepackage[T1]{fontenc}
\usepackage[latin9]{inputenc}
\usepackage{amstext}
\usepackage{amsthm}
\usepackage{amssymb}

\makeatletter
\numberwithin{equation}{section}
\numberwithin{figure}{section}
\theoremstyle{plain}
\newtheorem{thm}{\protect\theoremname}
  \theoremstyle{remark}
  \newtheorem{rem}[thm]{\protect\remarkname}
  \theoremstyle{definition}
  \newtheorem{defn}[thm]{\protect\definitionname}
  \theoremstyle{plain}
  \newtheorem{prop}[thm]{\protect\propositionname}
  \theoremstyle{plain}
  \newtheorem{lem}[thm]{\protect\lemmaname}
  \theoremstyle{plain}
  \newtheorem{cor}[thm]{\protect\corollaryname}

\makeatother

\usepackage{babel}
  \providecommand{\corollaryname}{Corollary}
  \providecommand{\definitionname}{Definition}
  \providecommand{\lemmaname}{Lemma}
  \providecommand{\propositionname}{Proposition}
  \providecommand{\remarkname}{Remark}
\providecommand{\theoremname}{Theorem}

\begin{document}

\title[A compactness result for low dimensional manifolds]{A compactness result for scalar-flat metrics on low dimensional manifolds
with umbilic boundary }

\author{Marco G. Ghimenti}
\address{M. G. Ghimenti, \newline Dipartimento di Matematica Universit\`a di Pisa
Largo B. Pontecorvo 5, 56126 Pisa, Italy}
\email{marco.ghimenti@unipi.it}

\author{Anna Maria Micheletti}
\address{A. M. Micheletti, \newline Dipartimento di Matematica Universit\`a di Pisa
Largo B. Pontecorvo 5, 56126 Pisa, Italy}
\email{a.micheletti@dma.unipi.it.}

\begin{abstract}
Let $(M,g)$ a compact Riemannian $n$-dimensional manifold with umbilic
boundary. It is well know that, under certain hypothesis, in the conformal
class of $g$ there are scalar-flat metrics that have $\partial M$
as a constant mean curvature hypersurface. In this paper we prove
that these metrics are a compact set in the case of low dimensional
manifolds, that is $n=6,7,8$, provided that the Weyl tensor is always
not vanishing on the boundary.
\end{abstract}


\maketitle

\section{Introduction}

Let $(M,g)$ be a $n$-dimensional ($n\ge3$) compact Riemannian manifold
with boundary $\partial M$. In \cite{Es,Es2} J. Escobar investigated
the question if $M$ can be conformally deformed to a scalar flat
manifold with boundary of constant mean curvature hypersurface. This
problem is particularly interesting because it is a higher-dimensional
generalization of the well known Riemann mapping Theorem and it is
equivalent to finding positive solutions to a linear equation on the
interior of $M$ with a critical nonlinear boundary condition of Neumann
type:
\begin{equation}
\left\{ \begin{array}{cc}
L_{g}u=0 & \text{ in }M\\
B_{g}u+(n-2)u^{\frac{n}{n-2}}=0 & \text{ on }\partial M
\end{array}\right..\label{eq:Pconf}
\end{equation}
Here $L_{g}=\Delta_{g}-\frac{n-2}{4(n-1)}R_{g}$ where $-\Delta_{g}$
is the Laplace-Beltrami operator on $(M,g)$ and $R_{g}$ the scalar
curvature of $M$ and $B_{g}=-\frac{\partial}{\partial\nu}-\frac{n-2}{2}h_{g}$,
where $\nu$ is the outward normal to $\partial M$ and $h_{g}$ is
the mean curvature of the boundary. 

The existence of solutions in established by Escobar \cite{Es}, Marques
\cite{M1}, Almaraz \cite{A3}, Chen \cite{ch}, Mayer and Ndiaye
\cite{MN}. Once the existence of solutions of (\ref{eq:Pconf}) is
settled, it is natural to study the compactness of the full set of
solutions. Defined 
\[
Q(M,\partial M):=\inf\left\{ Q(u)\ :\ u\in H^{1}(M),\ u\not\equiv0\text{ on }\partial M\right\} ,
\]
where
\[
Q(u):=\frac{\int\limits _{M}\left(|\nabla u|^{2}+\frac{n-2}{4(n-1)}R_{g}u^{2}\right)dv_{g}+\int\limits _{\partial M}\frac{n-2}{2}h_{g}u^{2}d\sigma_{g}}{\left(\int\limits _{\partial M}|u|^{\frac{2(n-1)}{n-2}}d\sigma_{g}\right)^{\frac{n-2}{n-1}}},
\]
we have that when $Q(M,\partial M)\le0$ the solution is unique up
to a constant factor. The situation turns out to be delicate if $Q(M,\partial M)>0$
and the underlying manifold is not the euclidean ball (in the case
of the euclidean ball the set of solution is known to be non compact).
Compactness has be proven firstly by Felli and Ould Ahmedou in \cite{FA}
for any dimension $n\ge3$ in the case of locally conformally flat
manifolds with umbilic boundary. If the dimension of the manifold
is $n\ge7$ and the trace-free second fundamental form in non zero
everywhere on $\partial M$, Almaraz in \cite{Al} proved compactness.
Very recently, Kim Musso and Wei \cite{KMW} showed that compactness
continues to hold when $n=4$ and when $n=6,7$ and the trace-free
second fundamental form in non zero everywhere on $\partial M$.

Compactness was proved also by the authors in \cite{GM} for manifold
with umbilic boundary when $n=8$ and the Weyl tensor of the boundary
is always different from zero, or if $n>8$ and the Weyl tensor of
$M$ is always different from zero on the boundary. An example of
non compactness is given for $n\ge25$ and manifolds with umbilic
boundary in \cite{A2}. We recall that the boundary of $M$ is called
\emph{umbilic} if the trace-free second fundamental form of $\partial M$
is zero everywhere.

In the present work we are interested to extend the result of \cite{GM}
to dimension $n=6,7,8$ when the Weyl tensor of $M$ is always different
from zero on the boundary. Namely we want to prove compactness of
the set of positive solutions to
\begin{equation}
\left\{ \begin{array}{cc}
L_{g}u=0 & \text{ in }M\\
B_{g}u+(n-2)u^{p}=0 & \text{ on }\partial M
\end{array}\right.\label{eq:Prob-2}
\end{equation}
where $1\le p\le\frac{n}{n-2}$ and the boundary of $M$ is umbilic.
Our main result is the following.
\begin{thm}
\label{thm:main}Let $(M,g)$ a smooth, $n$-dimensional Riemannian
manifold of positive type with regular umbilic boundary $\partial M$.
Suppose that $n=6,7,8$ and that the Weyl tensor $W_{g}$ is not vanishing
on $\partial M$. Then, given $\bar{p}>1$, there exists a positive
constant $C$ such that, for any $p\in\left[\bar{p},\frac{n}{n-2}\right]$
and for any $u>0$ solution of (\ref{eq:Prob-2}), it holds 
\[
C^{-1}\le u\le C\text{ and }\|u\|_{C^{2,\alpha}(M)}\le C
\]
for some $0<\alpha<1$. The constant $C$ does not depend on $u,p$. 
\end{thm}
Our strategy follows the argument of the seminal paper of Khuri Marques
and Schoen \cite{KMS}. A crucial step is to provide a sharp correction
term (see Subsection \ref{subsec:gamma}) for the usual approximation
of a rescaled solution by a bubble around an isolated simple blow
up point. This sharp correction term is a solution of a suitable linearized
equation (see (\ref{eq:vqdef})). The assumption of the umbilicity
of the boundary forces us to deal to higher order terms in the expansion
of the metric tensor, and this makes the proof of the result technically
hard. Moreover, it determines the right hand side of the equation
(\ref{eq:vqdef}), which gives the aforementioned correction term.

Another crucial step relies on a classical local argument with a Pohozaev
type identity and we need a local Pohozaev sign condition which is
essential for the proof. In the case of low dimensional manifolds
this requires a very accurate pointwise estimate of the correction
term which seems not to have an explicit form in the case of boundary
Yamabe problem. This process is somewhat inspired to the strategy
used by Kim Musso and Wei \cite{KMW} to estimate the correction term
on low dimensional manifold with non umbilic boundary.

The paper is organized as follows: in Section \ref{sec:Prel} we provide
some necessary preliminary notions; in particular in Subsection \ref{subsec:KMS}
we introduce some type of blow up points and in Subsection \ref{subsec:gamma}
we define the correction term. Section \ref{sec:characterization}
contains an accurate description of the correction term, and the Pohozaev
sign condition is studied in Section \ref{sec:7}, for the case $n=7,8$,
and in Section \ref{sec:6}, for the case $n=6$. The proof of Theorem
\ref{thm:main} is shown in Section \ref{sec:Proof}. Some technical
proofs are postponed to the Appendix.

\section{\label{sec:Prel}Preliminaries and notations}
\begin{rem}
We collect here our main notations. We will use the indices $1\le i,j,k,m,p,r,s,t,\tau\le n-1$
and $1\le a,b,c,d\le n$. Moreover we use the Einstein convention
on repeated indices. We denote by $g$ the Riemannian metric, by $R_{abcd}$
the full Riemannian curvature tensor, by $R_{ab}$ the Ricci tensor
and by $R_{g}$ the scalar curvature of $(M,g)$; moreover the Weyl
tensor of $(M,g)$ will be denoted by $W_{g}$. The bar over an object
(e.g. $\bar{W}_{g}$) will means the restriction to this object to
the metric of $\partial M$. Finally, on the half space $\mathbb{R}_{+}^{n}=\left\{ y=(y_{1},\dots,y_{n-1},y_{n})\in\mathbb{R}^{n},\!\ y_{n}\ge0\right\} $
we set $B_{r}(y_{0})=\left\{ y\in\mathbb{R}^{n},\!\ |y-y_{0}|\le r\right\} $
and $B_{r}^{+}(y_{0})=B_{r}(y_{0})\cap\left\{ y_{n}>0\right\} $.
When $y_{0}=0$ we will use simply $B_{r}=B_{r}(y_{0})$ and $B_{r}^{+}=B_{r}^{+}(y_{0})$.
On the half ball $B_{r}^{+}$ we set $\partial'B_{r}^{+}=B_{r}^{+}\cap\partial\mathbb{R}_{+}^{n}=B_{r}^{+}\cap\left\{ y_{n}=0\right\} $
and $\partial^{+}B_{r}^{+}=\partial B_{r}^{+}\cap\left\{ y_{n}>0\right\} $.
On $\mathbb{R}_{+}^{n}$ we will use the following decomposition of
coordinates: $(y_{1},\dots,y_{n-1},y_{n})=(\bar{y},y_{n})=(z,t)$
where $\bar{y},z\in\mathbb{R}^{n-1}$ and $y_{n},t\ge0$.

Fixed a point $q\in\partial M$, we denote by $\psi_{q}:B_{r}^{+}\rightarrow M$
the Fermi coordinates centered at $q$. We denote by $B_{g}^{+}(q,r)$
the image of $\psi_{q}(B_{r}^{+})$. When no ambiguity is possible,
we will denote $B_{g}^{+}(q,r)$ simply by $B_{r}^{+}$, omitting
the chart $\psi_{q}$.

We recall that $\omega_{n-2}$ is the $n-1$ dimensional spherical
element.
\end{rem}
Since the boundary $\partial M$ of $M$ is umbilic, it is well know
the existence of a conformal metric related to $g$ and the existence
of the conformal Fermi coordinates, which will simplify the future
computations. 

Given $q\in\partial M$ there exists a conformally related metric
$\tilde{g}_{q}=\Lambda_{q}g$ such that some geometric quantities
at $q$ have a simpler form which will be summarized in the next claim.
We also know that $\Lambda_{q}(q)=1,\ \frac{\partial\Lambda_{q}}{\partial y_{k}}(q)=0\text{ for all }k=1,\dots,n-1.$
In order to simplify notations, we will omit the \emph{tilde} symbol
and we will omit the fermi conformal coordinates $\psi_{q}:B_{r}^{+}\rightarrow M$
whenever it is not needed, so we will write $y\in B_{r}^{+}$instead
of $\psi_{q}(y)\in M$, $0$ instead of $q=\psi_{q}(0)$, $u$ instead
of $u\circ\psi_{q}$ and so on.
\begin{rem}
\label{rem:confnorm}In Fermi conformal coordinates around $q\in\partial M$,
it holds (see \cite{M1})
\begin{equation}
|\text{det}g_{q}(y)|=1+O(|y|^{N})\text{ for some }N\text{ large}\label{eq:|g|}
\end{equation}
\begin{eqnarray}
|h_{ij}(y)|=O(|y^{4}|) &  & |h_{g}(y)|=O(|y^{4}|)\label{eq:hij}
\end{eqnarray}
\begin{align}
g_{q}^{ij}(y)= & \delta^{ij}+\frac{1}{3}\bar{R}_{ikjl}y_{k}y_{l}+R_{ninj}y_{n}^{2}\label{eq:gij}\\
 & +\frac{1}{6}\bar{R}_{ikjl,m}y_{k}y_{l}y_{m}+R_{ninj,k}y_{n}^{2}y_{k}+\frac{1}{3}R_{ninj,n}y_{n}^{3}\nonumber \\
 & +\left(\frac{1}{20}\bar{R}_{ikjl,mp}+\frac{1}{15}\bar{R}_{iksl}\bar{R}_{jmsp}\right)y_{k}y_{l}y_{m}y_{p}\nonumber \\
 & +\left(\frac{1}{2}R_{ninj,kl}+\frac{1}{3}\text{Sym}_{ij}(\bar{R}_{iksl}R_{nsnj})\right)y_{n}^{2}y_{k}y_{l}\nonumber \\
 & +\frac{1}{3}R_{ninj,nk}y_{n}^{3}y_{k}+\frac{1}{12}\left(R_{ninj,nn}+8R_{nins}R_{nsnj}\right)y_{n}^{4}+O(|y|^{5})\nonumber 
\end{align}
\begin{equation}
\bar{R}_{g_{q}}(y)=O(|y|^{2})\text{ and }\partial_{ii}^{2}\bar{R}_{g_{q}}=-\frac{1}{6}|\bar{W}|^{2}\label{eq:Rii}
\end{equation}
\begin{equation}
\partial_{tt}^{2}\bar{R}_{g_{q}}=-2R_{ninj}^{2}-2R_{ninj,ij}\label{eq:Rtt}
\end{equation}
\begin{equation}
\bar{R}_{kl}=R_{nn}=R_{nk}=R_{nn,kk}=0\label{eq:Ricci}
\end{equation}
\begin{equation}
R_{nn,nn}=-2R_{nins}^{2}.\label{eq:Rnnnn}
\end{equation}
All the quantities above are calculate in $q\in\partial M$, unless
otherwise specified.
\end{rem}
We set ${\displaystyle U(y):=\frac{1}{\left[(1+y_{n})^{2}+|\bar{y}|^{2}\right]^{\frac{n-2}{2}}}}$
to be the standard bubble. The function $U$ solves the problem 
\begin{equation}
\left\{ \begin{array}{cc}
\Delta U=0 & \text{ in }\mathbb{R}_{+}^{n}\\
\frac{\partial U}{\partial y_{n}}+(n-2)U^{\frac{n}{n-2}}=0 & \text{ on }\partial\mathbb{R}_{+}^{n}
\end{array}\right..\label{eq:ProbBubble}
\end{equation}
\begin{rem}
Let $f:\mathbb{R}\times\mathbb{R}^{+}\rightarrow\mathbb{R}$ be a
smooth integrable function and fix a $c\ge0$. We have the following
integral identities
\begin{equation}
R_{ninj}\int_{\partial\mathbb{R}_{+}^{n}}f(|\bar{y}|,c)y_{i}y_{j}d\bar{y}=0\label{eq:Sym1}
\end{equation}
\begin{equation}
\bar{R}_{t\tau sp}\int_{\mathbb{\partial R}_{+}^{n}}f(|\bar{y}|,c)y_{t}y_{\tau}y_{s}y_{p}d\bar{y}=0\label{eq:Sym2}
\end{equation}
\begin{equation}
R_{ninj}\bar{R}_{t\tau sp}\int_{\partial\mathbb{R}_{+}^{n}}f(|\bar{y}|,c)y_{i}y_{j}y_{t}y_{\tau}y_{s}y_{p}d\bar{y}=0\label{eq:Sym3}
\end{equation}

\begin{equation}
\bar{R}_{ijkl}\bar{R}_{t\tau sp}\int_{\mathbb{\partial R}_{+}^{n}}f(|\bar{y}|,c)y_{i}y_{j}y_{k}y_{l}y_{t}y_{\tau}y_{s}y_{p}d\bar{y}=0\label{eq:Sym4}
\end{equation}
\begin{align}
R_{ninj}R_{nknl}\int_{\partial\mathbb{R}_{+}^{n}}f(|\bar{y}|,c)y_{i}y_{j}y_{k}y_{l}d\bar{y} & =\frac{2}{3}R_{ninj}^{2}\int_{\partial\mathbb{R}_{+}^{n}}f(|\bar{y}|,c)y_{1}^{4}d\bar{y}\label{eq:Sym5}\\
 & =\frac{2}{n^{2}-1}R_{ninj}^{2}\int_{\partial\mathbb{R}_{+}^{n}}f(|\bar{y}|,c)|\bar{y}|^{4}d\bar{y}\nonumber 
\end{align}
 
\end{rem}
\begin{proof}
The first two identities follows by the symmetries of the curvature
tensor. For the last formula we have, again by symmetry,
\begin{align*}
R_{ninj}R_{nknl}\int_{\partial\mathbb{R}_{+}^{n}}f(|\bar{y}|,c)y_{i}y_{j}y_{k}y_{l}d\bar{y} & =2R_{ninj}^{2}\int_{\partial\mathbb{R}_{+}^{n}}f(|\bar{y}|,c)y_{i}^{2}y_{j}^{2}d\bar{y}\\
 & =2R_{ninj}^{2}\int_{\partial\mathbb{R}_{+}^{n}}f(|\bar{y}|,c)y_{1}^{2}y_{2}^{2}d\bar{y},
\end{align*}
and we can conclude by the elementary identities
\[
3\int_{\mathbb{R}^{2}}f(x^{2}+y^{2})x^{2}y^{2}dxdy=\int_{\mathbb{R}^{2}}f(x^{2}+y^{2})x^{4}dxdy
\]
and 
\[
\int_{\partial\mathbb{R}_{+}^{n}}f(|\bar{y}|,c)\bar{y}_{1}^{4}d\bar{y}=\frac{3}{n^{2}-1}\int_{\partial\mathbb{R}_{+}^{n}}f(|\bar{y}|,c)|\bar{y}|^{4}d\bar{y}.
\]
\end{proof}
\begin{rem}
\label{rem:Iam}We collect here some result contained in \cite[Lemma 9.4]{Al}
and in \cite[Lemma 9.5]{Al}. The proof is by direct computation.
For $m>k+1$
\begin{align}
\int_{0}^{\infty}\frac{t^{k}dt}{(1+t)^{m}} & =\frac{k!}{(m-1)(m-2)\cdots(m-1-k)}\label{eq:t-integrali}\\
\int_{0}^{\infty}\frac{dt}{(1+t)^{m}} & =\frac{1}{m-1}\nonumber 
\end{align}
Moreover, set, for $\alpha,m\in\mathbb{N}$, 
\[
I_{m}^{\alpha}:=\int_{0}^{\infty}\frac{s^{\alpha}ds}{\left(1+s^{2}\right)^{m}}
\]
it holds
\begin{align}
I_{m}^{\alpha}=\frac{2m}{\alpha+1}I_{m+1}^{\alpha+2} & \text{ for }\alpha+1<2m\label{eq:Iam}\\
I_{m}^{\alpha}=\frac{2m}{2m-\alpha-1}I_{m+1}^{\alpha} & \text{ for }\alpha+1<2m\nonumber \\
I_{m}^{\alpha}=\frac{2m-\alpha-3}{\alpha+1}I_{m}^{\alpha+2} & \text{ for }\alpha+3<2m.\nonumber 
\end{align}
\end{rem}

\subsection{\label{subsec:KMS}Blow up points and the Khuri-Marques-Schoen scheme}

By the conformal invariance property of the operators $L_{g}$ and
$B_{g}$ it is more convenient to deal with the conformally invariant
family of problems

\begin{equation}
\left\{ \begin{array}{cc}
L_{g_{i}}u=0 & \text{ in }M\\
B_{g_{i}}u+(n-2)f_{i}^{-\tau_{i}}u^{p_{i}}=0 & \text{ on }\partial M
\end{array}\right..\label{eq:Prob-i}
\end{equation}
where $p_{i}\in\left[\bar{p},\frac{n}{n-2}\right]$ for some fixed
$\bar{p}>1$, $\tau_{i}=\frac{n}{n-2}-p_{i}$, $f_{i}\rightarrow f$
in $C_{\text{loc}}^{1}$ for some positive function $f$ and $g_{i}\rightarrow g_{0}$
in the $C_{\text{loc}}^{3}$ topology.

First, we collect the definition of various type of blow up points.
\begin{defn}
\label{def:blowup}We say that $x_{0}\in\partial M$ is a blow up
point for the sequence $u_{i}$ of solutions of (\ref{eq:Prob-i})
if there is a sequence $x_{i}\in\partial M$ such that 
\begin{enumerate}
\item $x_{i}\rightarrow x_{0}$;
\item $x_{i}$ is a local maximum point of $\left.u_{i}\right|_{\partial M}$
;
\item $u_{i}(x_{i})\rightarrow+\infty.$
\end{enumerate}
Shortly we say that $x_{i}\rightarrow x_{0}$ is a blow up point for
$\left\{ u_{i}\right\} _{i}$. 

We say that $x_{i}\rightarrow x_{0}$ is an \emph{isolated} blow up
point for $\left\{ u_{i}\right\} _{i}$ if $x_{i}\rightarrow x_{0}$
is a blow up point for $\left\{ u_{i}\right\} _{i}$ and there exist
two constants $\rho,C>0$ such that
\[
u_{i}(x)\le Cd_{\bar{g}}(x,x_{i})^{-\frac{1}{p_{i-1}}}\text{ for all }x\in\partial M\smallsetminus\left\{ x_{i}\right\} ,\ d_{\bar{g}}(x,x_{i})<\rho.
\]
Here $\bar{g}$ denotes the metric on the boundary induced by $g$
and $d_{\bar{g}}(\cdot,\cdot)$ is the geodesic distance on the boundary
between two points.

Finally, given $x_{i}\rightarrow x_{0}$ an isolated blow up point
for $\left\{ u_{i}\right\} _{i}$, and given $\psi_{i}:B_{\rho}^{+}(0)\rightarrow M$
the Fermi coordinates centered at $x_{i}$, we define the spherical
average of $u_{i}$ as
\[
\bar{u}_{i}(r)=\frac{2}{\omega_{n-1}r^{n-1}}\int_{\partial^{+}B_{r}^{+}}u_{i}\circ\psi_{i}d\sigma_{r}
\]
and
\[
w_{i}(r):=r^{-\frac{1}{p_{i}-1}}\bar{u}_{i}(r)
\]
for $0<r<\rho.$

We say that $x_{i}\rightarrow x_{0}$ is an \emph{isolated simple}
blow up point for $\left\{ u_{i}\right\} _{i}$ solutions of (\ref{eq:Prob-i})
if $x_{i}\rightarrow x_{0}$ is an isolated blow up point for $\left\{ u_{i}\right\} _{i}$
and there exists $\rho$ such that $w_{i}$ has exactly one critical
point in the interval $(0,\rho)$. 
\end{defn}
It is possible to prove the following proposition (see, for example
\cite{Al,FA,GM,KMS})
\begin{prop}
Let $x_{i}\rightarrow x_{0}$ is an isolated blow up point for $\left\{ u_{i}\right\} _{i}$
and $\rho$ as in Definition \ref{def:blowup}. We set 
\[
v_{i}(y)=M_{i}^{-1}(u_{i}\circ\psi_{i})(M_{i}^{1-p_{i}}y),\text{ for }y\in B_{\rho M_{i}^{p_{i}-1}}^{+}(0),\text{ where }M_{i}:=u_{i}(x_{i})
\]
 Then, given $R_{i}\rightarrow\infty$ and $\beta_{i}\rightarrow0$,
up to subsequences, we have 
\[
|v_{i}-U|_{C^{2}\left(B_{R_{i}}^{+}(0)\right)}<\beta_{i}\text{ and }\lim_{i\rightarrow\infty}p_{i}=\frac{n}{n-2}.
\]

Furthermore, if $x_{i}\rightarrow x_{0}$ is an isolated simple blow
up point for $\left\{ u_{i}\right\} _{i}$, then there exist $C,\rho>0$
such that
\begin{enumerate}
\item $M_{i}u_{i}(\psi_{i}(y))\le C|y|^{2-n}$ for all $y\in B_{\rho}^{+}(0)\smallsetminus\left\{ 0\right\} $; 
\item $M_{i}u_{i}(\psi_{i}(y))\ge C^{-1}G_{i}(y)$ for all $y\in B_{\rho}^{+}(0)\smallsetminus B_{r_{i}}^{+}(0)$
where $r_{i}:=R_{i}M_{i}^{1-p_{i}}$ and $G_{i}$ is the Green's function
which solves 
\[
\left\{ \begin{array}{ccc}
L_{g_{i}}G_{i}=0 &  & \text{in }B_{\rho}^{+}(0)\smallsetminus\left\{ 0\right\} \\
G_{i}=0 &  & \text{on }\partial^{+}B_{\rho}^{+}(0)\\
B_{g_{i}}G_{i}=0 &  & \text{on }\partial'B_{\rho}^{+}(0)\smallsetminus\left\{ 0\right\} 
\end{array}\right.
\]
and $|y|^{n-2}G_{i}(y)\rightarrow1$ as $|y|\rightarrow0$. 
\end{enumerate}
\end{prop}
The usual strategy to prove compactness of solutions of Yamabe problems
dates back to the seminal Khuri Marques and Schoen paper \cite{KMS}.
Their idea is to prove firstly that only isolated simple blow up points
may occur, then, to give a precise description of the asymptotic profile
of a rescaled solution around an isolated simple blow up points. Finally
they rule out also the possibility of having isolated simple blow
up points.

The key tool to accomplish these steps is a sign estimates of a Pohozaev
type formula for a blowing up sequence of solutions that we recall
here.
\begin{thm}[Pohozaev Identity]
\label{thm:poho} Let $u$ a $C^{2}$-solution of the following problem
\[
\left\{ \begin{array}{cc}
L_{g}u=0 & \text{ in }B_{r}^{+}\\
B_{g}u+(n-2)f^{-\tau}u^{p}=0 & \text{ on }\partial'B_{r}^{+}
\end{array}\right.
\]
for $B_{r}^{+}=\psi_{q}^{-1}(B_{g}^{+}(q,r))$ for $q\in\partial M$,
with $\tau=\frac{n}{n-2}-p>0$. Let us define 
\[
\bar{P}(u,r):=\int\limits _{\partial^{+}B_{r}^{+}}\left(\frac{n-2}{2}u\frac{\partial u}{\partial r}-\frac{r}{2}|\nabla u|^{2}+r\left|\frac{\partial u}{\partial r}\right|^{2}\right)d\sigma_{r}+\frac{r(n-2)}{p+1}\int\limits _{\partial(\partial'B_{r}^{+})}f^{-\tau}u^{p+1}d\bar{\sigma}_{g}
\]
and
\begin{multline*}
P(u,r)=-\int\limits _{B_{r}^{+}}\left(y^{a}\partial_{a}u+\frac{n-2}{2}u\right)[(L_{g}-\Delta)u]dy+\frac{n-2}{2}\int\limits _{\partial'B_{r}^{+}}\left(\bar{y}^{k}\partial_{k}u+\frac{n-2}{2}u\right)h_{g}ud\bar{y}\\
-\frac{\tau(n-2)}{p+1}\int\limits _{\partial'B_{r}^{+}}\left(\bar{y}^{k}\partial_{k}f\right)f^{-\tau-1}u^{p+1}d\bar{y}+\left(\frac{n-1}{p+1}-\frac{n-2}{2}\right)\int\limits _{\partial'B_{r}^{+}}(n-2)f^{-\tau}u^{p+1}d\bar{y}.
\end{multline*}
Then 
\[
\bar{P}(u,r)=P(u,r)
\]
\end{thm}

\subsection{\label{subsec:gamma}A sharp approximation of blow up points}

To describe the asymptotic profile of a rescaled solution around an
isolated simple blow up point in the case of manifolds with umbilic
boundary we introduce the function $\gamma_{q}=\gamma$ which solves
\begin{equation}
\left\{ \begin{array}{ccc}
-\Delta\gamma=\left[\frac{1}{3}\bar{R}_{ikjl}(q)y_{k}y_{l}+R_{ninj}(q)y_{n}^{2}\right]\partial_{ij}^{2}U &  & \text{on }\mathbb{R}_{+}^{n}\\
\frac{\partial\gamma}{\partial y_{n}}=-nU^{\frac{2}{n-2}}\gamma &  & \text{on }\partial\mathbb{R}_{+}^{n}
\end{array}\right..\label{eq:vqdef}
\end{equation}
In \cite{GMP} and in \cite{GM} the authors prove the following lemma.
\begin{lem}
\label{lem:vq}Assume $n\ge5$. Given a point $q\in\partial M$, there
exists a solution $\gamma:\mathbb{R}_{+}^{n}\rightarrow\mathbb{R}$
of the linear problem (\ref{eq:vqdef}).

In addition it holds
\begin{equation}
|\nabla^{\tau}\gamma(y)|\le C(1+|y|)^{4-\tau-n}\text{ for }\tau=0,1,2;\label{eq:gradvq}
\end{equation}
\begin{equation}
\int_{\mathbb{R}_{+}^{n}}\gamma\Delta\gamma dy\le0;\label{new}
\end{equation}

\begin{equation}
\int_{\partial\mathbb{R}_{+}^{n}}U^{\frac{n}{n-2}}(t,z)\gamma(t,z)dz=0;\label{eq:Uvq}
\end{equation}
\begin{equation}
\gamma(0)=\frac{\partial\gamma}{\partial y_{1}}(0)=\dots=\frac{\partial\gamma}{\partial y_{n-1}}(0)=0.\label{eq:dervq}
\end{equation}
\end{lem}
Let $x_{i}\rightarrow x_{0}$ an isolated simple blow up point for
$u_{i}$ of solutions of (\ref{eq:Prob-i}) . Set 
\[
v_{i}(y):=\delta_{i}^{\frac{1}{p_{i}-1}}u_{i}(\delta_{i}y)\text{ for }y\in B_{\frac{R}{\delta_{i}}}^{+}(0)\text{ where }\delta_{i}:=u_{i}^{1-p_{i}}(x_{i}),
\]
we know that $v_{i}$ satisfies 
\begin{equation}
\left\{ \begin{array}{cc}
L_{\hat{g_{i}}}v_{i}=0 & \text{ in }B_{\frac{R}{\delta_{i}}}^{+}(0)\\
B_{\hat{g_{i}}}v_{i}+(n-2)\hat{f}^{-\tau_{i}}v_{i}^{p_{i}}=0 & \text{ on }\partial B_{\frac{R}{\delta_{i}}}^{+}(0)
\end{array}\right.\label{eq:Prob-hat}
\end{equation}
where $\hat{g}_{i}:=\tilde{g}_{i}(\delta_{i}y)=\Lambda_{x_{i}}^{\frac{4}{n-2}}(\delta_{i}y)g(\delta_{i}y)$,
$\hat{f}_{i}(y)=f_{i}(\delta_{i}y)$, $f_{i}=\Lambda_{x_{i}}f\rightarrow\Lambda_{x_{0}}f$
and $\tau_{i}=\frac{n}{n-2}-p_{i}$.

Using the term $\gamma$ we are able to give a good estimate of the
rescaled solution $v_{i}$ around the isolated blow up point $x_{i}\rightarrow x_{0}$.
Indeed we have (see \cite[Proposition 9]{GM})
\begin{prop}
\label{prop:stimawi}Assume $n\ge6$. Let $\gamma$ be defined in
(\ref{eq:vqdef}). There exist $R,C>0$ such that 
\begin{align*}
|v_{i}(y)-U(y)-\delta_{i}^{2}\gamma_{x_{i}}(y)| & \le C\delta_{i}^{3}(1+|y|)^{5-n}\\
\left|\frac{\partial}{\partial_{j}}\left(v_{i}(y)-U(y)-\delta_{i}^{2}\gamma_{x_{i}}(y)\right)\right| & \le C\delta_{i}^{3}(1+|y|)^{4-n}\\
\left|y_{n}\frac{\partial}{\partial_{n}}\left(v_{i}(y)-U(y)-\delta_{i}^{2}\gamma_{x_{i}}(y)\right)\right| & \le C\delta_{i}^{3}(1+|y|)^{5-n}\\
\left|\frac{\partial^{2}}{\partial_{j}\partial_{k}}\left(v_{i}(y)-U(y)-\delta_{i}^{2}\gamma_{x_{i}}(y)\right)\right| & \le C\delta_{i}^{3}(1+|y|)^{3-n}
\end{align*}
for $|y|\le\frac{R}{2\delta_{i}}$. 
\end{prop}

\section{\label{sec:characterization}A characterization of function $\gamma$}

In this section we give a an accurate description of a solution $\gamma$
of (\ref{eq:vqdef}), similarly to \cite{KMW}. First we split
\[
\gamma=\Phi+E
\]
where $\Phi=\tilde{\Phi}_{1}+\tilde{\Phi}_{2}$ is a polynomial function
and $\tilde{\Phi}_{1},\tilde{\Phi}_{2}$ solve, respectively
\begin{align}
-\Delta\tilde{\Phi}_{1}=R_{ninj}(q)y_{n}^{2}\partial_{ij}^{2}U & \text{on }\mathbb{R}_{+}^{n}\label{eq:Phidef1}\\
-\Delta\tilde{\Phi}_{2}=\frac{1}{3}\bar{R}_{ijkl}(q)y_{k}y_{l}\partial_{ij}^{2}U & \text{on }\mathbb{R}_{+}^{n}\label{eq:Phidef2}
\end{align}
while $E$ is an harmonic function solving 
\begin{equation}
\left\{ \begin{array}{ccc}
-\Delta E=0 &  & \text{on }\mathbb{R}_{+}^{n}\\
\lim\limits _{y_{n}\rightarrow0}\frac{\partial E}{\partial y_{n}}=-nU^{\frac{2}{n-2}}E-q &  & \text{on }\partial\mathbb{R}_{+}^{n}
\end{array}\right.,\label{eq:E}
\end{equation}
with $q=\frac{\partial\Phi}{\partial y_{n}}+nU^{\frac{2}{n-2}}\Phi$.
\begin{lem}
\label{lem:Phitilda2} For $n=5$ or $n\ge7$ the function
\[
\tilde{\Phi}_{2}=\frac{1}{3}\bar{R}_{ijkl}(q)y_{i}y_{j}y_{k}y_{l}\left\{ \frac{n-2}{6(|\bar{y}|^{2}+(1+y_{n})^{2})^{\frac{n}{2}}}+a_{1}\frac{n(n^{2}-4)(n+4)}{(n-6)(n-4)}\frac{1}{(|\bar{y}|^{2}+(1+y_{n})^{2})^{\frac{n+6}{2}}}\right\} 
\]
solves (\ref{eq:Phidef2}) for any $a_{1}\in\mathbb{R}$.
\end{lem}
\begin{lem}
\label{lem:Phitilda1}For $n=5$ or $n\ge7$ the function
\begin{align*}
\tilde{\Phi}_{1}= & R_{ninj}(q)y_{i}y_{j}\left\{ \frac{1}{12(|\bar{y}|^{2}+(1+y_{n})^{2})^{\frac{n-2}{2}}}+\frac{n-2}{6}\frac{1+y_{n}^{2}-y_{n}}{(|\bar{y}|^{2}+(1+y_{n})^{2})^{\frac{n}{2}}}\right.\\
 & +a_{1}\frac{n(n^{2}-4)}{(n-4)(n-6)}\left[(n+4)\frac{(1+y_{n}^{2})}{(|\bar{y}|^{2}+(1+y_{n})^{2})^{\frac{n+6}{2}}}-\frac{1}{(|\bar{y}|^{2}+(1+y_{n})^{2})^{\frac{n+4}{2}}}\right]\\
 & +a'_{1}\left[\frac{n(n-2)}{n-4}\frac{1}{(|\bar{y}|^{2}+(1+y_{n})^{2})^{\frac{n+2}{2}}}-2n(n+2)\frac{1}{(|\bar{y}|^{2}+(1+y_{n})^{2})^{\frac{n+4}{2}}}\right]\\
 & \left.+a'_{2}n(n-2)\left[\frac{1}{(|\bar{y}|^{2}+(1+y_{n})^{2})^{\frac{n+2}{2}}}\right]\right\} .
\end{align*}
solves (\ref{eq:Phidef1}) for any $a_{1},a'_{1},a'_{2}\in\mathbb{R}$.
\end{lem}
The proof of these two results is postponed in the appendix.

For our purpose will be sufficient to fix $a_{1}=a'_{1}=0$. This
allows also to extend the previous results for $n=6$, as we summarize
hereafter.
\begin{cor}
\label{cor:Phi1e2}For $n\ge5$ the functions
\begin{align*}
\tilde{\Phi}_{1}:= & R_{ninj}(q)y_{i}y_{j}\left\{ \frac{1}{12(|\bar{y}|^{2}+(1+y_{n})^{2})^{\frac{n-2}{2}}}+\frac{n-2}{6}\frac{1+y_{n}^{2}-y_{n}}{(|\bar{y}|^{2}+(1+y_{n})^{2})^{\frac{n}{2}}}\right\} \\
\tilde{\Phi}_{2} & :=\frac{1}{3}\bar{R}_{ijkl}(q)y_{i}y_{j}y_{k}y_{l}\left\{ \frac{n-2}{6(|\bar{y}|^{2}+(1+y_{n})^{2})^{\frac{n}{2}}}\right\} 
\end{align*}
solve respectively (\ref{eq:Phidef1}) and (\ref{eq:Phidef2}).
\end{cor}
\begin{proof}
For $n=5$ and $n\ge7$ the result is proved in the appendix, in the
proofs of Lemmas \ref{lem:Phitilda2} and \ref{lem:Phitilda1}. For
$n=6$, notice that both functions $\tilde{\Phi}_{1},\tilde{\Phi}_{2}$
are well defined Then the claim follows by direct computation.
\end{proof}

\section{\label{sec:7}Case $n=7,8$}

In \cite{GM} it is proved that, if $x_{i}\rightarrow x_{0}$ is isolated
simple blow-up point for $u_{i}$, then, for $n\ge7$ it holds
\begin{align}
P(u_{i},r)\ge & R(U,U)+R(U,\delta_{i}^{2}\gamma)+R(\delta_{i}^{2}\gamma,U)+O(\delta^{n-2})\label{eq:poho1}\\
\ge & \delta_{i}^{4}\frac{(n-2)\omega_{n-2}I_{n}^{n}}{(n-1)(n-3)(n-5)(n-6)}\left[\frac{\left(n-2\right)}{6}|\bar{W}(x_{i})|^{2}+\frac{4(n-8)}{(n-4)}R_{nlnj}^{2}(x_{i})\right]\nonumber \\
 & -2\delta_{i}^{4}\int_{\mathbb{R}_{+}^{n}}\gamma_{x_{i}}\Delta\gamma_{x_{i}}dy+o(\delta_{i}^{4}).\nonumber 
\end{align}
where
\begin{equation}
R(u,v):=-\int_{B_{r/\delta_{i}}^{+}}\left(y^{b}\partial_{b}u+\frac{n-2}{2}u\right)\left[(L_{\hat{g}_{i}}-\Delta)v\right]dy.\label{eq:Ruv}
\end{equation}
and $\hat{g}_{i}:=\Lambda_{x_{i}}^{\frac{4}{n-2}}(\delta_{i}y)g(\delta_{i}y)$.

This, for $n=7$, becomes
\begin{align}
P(u_{i},r) & \ge\delta_{i}^{4}\omega_{5}I_{7}^{7}\left[\frac{25}{432}|\bar{W}(x_{i})|^{2}-\frac{5}{36}R_{7i7j}^{2}(x_{i})\right]-2\delta_{i}^{4}\int_{\mathbb{R}_{+}^{n}}\gamma\Delta\gamma dy+o(\delta_{i}^{4}),\label{eq:poho2}
\end{align}
and, for $n=8$, 

\begin{align}
P(u_{i},r) & \ge\frac{\delta_{i}^{4}\omega_{6}I_{8}^{8}}{35}|\bar{W}(x_{i})|^{2}-2\delta_{i}^{4}\int_{\mathbb{R}_{+}^{n}}\gamma\Delta\gamma dy+o(\delta_{i}^{4}),\label{eq:poho2-1}
\end{align}
The proof of (\ref{eq:poho1}) can be found in \cite[Prop. 14]{GM}. 

In this section we will prove the following result
\begin{lem}
\label{lem:poho7} Let $x_{i}\rightarrow x_{0}$ is an isolated simple
blow-up point for $u_{i}$ solution of (\ref{eq:Prob-i}) then it
holds
\begin{align}
P(u_{i},r) & \ge\delta_{i}^{4}\omega_{5}I_{7}^{7}\left[\frac{25}{432}|\bar{W}(x_{i})|^{2}+\frac{7}{54}R_{7i7j}^{2}(x_{i})\right]+o(\delta_{i}^{4})\text{ for }n=7;\label{eq:pohofinale7}\\
P(u_{i},r) & \ge\delta_{i}^{4}\omega_{6}I_{8}^{8}\left[\frac{1}{35}|\bar{W}(x_{i})|^{2}+\frac{1089}{34020}R_{8i8j}^{2}(x_{i})\right]+o(\delta_{i}^{4})\text{ for }n=8.\label{eq:pohofinale8}
\end{align}
\end{lem}

\subsection{A crucial estimate}

To prove Theorem \ref{thm:main} it will be necessary to estimate
the value of $-\int_{\mathbb{R}_{+}^{n}}\gamma\Delta\gamma d\bar{y}dy_{n}$
in order to obtain that the right hand sides of (\ref{eq:poho2})
and of (\ref{eq:poho2-1}) are positive. By the description of $\gamma$
in terms of $E$ and $\Phi$, we can simplify this integral term as
following.
\begin{lem}
\label{lem:gammaDgamma}We have 
\[
-\int_{\mathbb{R}_{+}^{n}}\gamma\Delta\gamma d\bar{y}dy_{n}=\int_{\partial\mathbb{R}_{+}^{n}}q\Phi d\bar{y}+\int_{\partial\mathbb{R}_{+}^{n}}qEd\bar{y}-\int_{\mathbb{R}_{+}^{n}}\Phi\Delta\Phi d\bar{y}dy.
\]
\end{lem}
\begin{proof}
We get, since $E$ is harmonic, and integrating by parts, that
\begin{align*}
-\int_{\mathbb{R}_{+}^{n}}\gamma\Delta\gamma d\bar{y}dy_{n} & =-\int_{\mathbb{R}_{+}^{n}}(E+\Phi)\Delta\Phi d\bar{y}dy_{n}\\
 & =\int_{\mathbb{R}_{+}^{n}}\nabla(E+\Phi)\nabla\Phi d\bar{y}dy_{n}+\int_{\partial\mathbb{R}_{+}^{n}}(E+\Phi)\partial_{n}\Phi d\bar{y}\\
 & =-\int_{\mathbb{R}_{+}^{n}}(\Delta\Phi)\Phi d\bar{y}dy_{n}-\int_{\partial\mathbb{R}_{+}^{n}}\partial_{n}(E+\Phi)\Phi d\bar{y}+\int_{\partial\mathbb{R}_{+}^{n}}(E+\Phi)\partial_{n}\Phi d\bar{y}\\
 & =-\int_{\mathbb{R}_{+}^{n}}(\Delta\Phi)\Phi d\bar{y}dy_{n}-\int_{\partial\mathbb{R}_{+}^{n}}\partial_{n}E\Phi d\bar{y}+\int_{\partial\mathbb{R}_{+}^{n}}E\partial_{n}\Phi d\bar{y}.
\end{align*}
Now, keeping in mind that $q=\frac{\partial\Phi}{\partial y_{n}}+nU^{\frac{2}{n-2}}\Phi$
and equation (\ref{eq:E}) we have
\[
-\int_{\partial\mathbb{R}_{+}^{n}}\partial_{n}E\Phi d\bar{y}+\int_{\partial\mathbb{R}_{+}^{n}}E\partial_{n}\Phi d\bar{y}=\int_{\partial\mathbb{R}_{+}^{n}}(nU^{\frac{2}{n-2}}E+q)\Phi d\bar{y}+\int_{\partial\mathbb{R}_{+}^{n}}E(q-nU^{\frac{2}{n-2}}\Phi)d\bar{y}
\]
and we get the result.
\end{proof}
\begin{lem}
\label{lem:qE}If $n>6$ we have 
\[
\int_{\partial\mathbb{R}_{+}^{n}}qEd\bar{y}=\int_{\mathbb{R}_{+}^{n}}|\nabla E|^{2}d\bar{y}dy_{n}-n\int_{\partial\mathbb{R}_{+}^{n}}U^{\frac{2}{n-2}}E^{2}d\bar{y}\ge0.
\]
\end{lem}
\begin{proof}
First of all, by (\ref{eq:E}), integrating by parts we have
\[
0=\int_{\mathbb{R}_{+}^{n}}-E\Delta Ed\bar{y}dy_{n}=\int_{\mathbb{R}_{+}^{n}}|\nabla E|^{2}d\bar{y}dy_{n}-n\int_{\partial\mathbb{R}_{+}^{n}}U^{\frac{2}{n-2}}E^{2}d\bar{y}-\int_{\partial\mathbb{R}_{+}^{n}}qEd\bar{y}
\]
which proves the first equality. Notice that $E\in D^{1,2}(\mathbb{R}_{+}^{n})$
by difference, since $\gamma,\Phi\in D^{1,2}(\mathbb{R}_{+}^{n})$
if $n>6$.

To conclude we argue as in \cite[Lemma 4.6]{KMW}. Firstly, observe
that, since $q=\frac{\partial(\tilde{\Phi}_{1}+\tilde{\Phi}_{2})}{\partial y_{n}}+nU^{\frac{2}{n-2}}(\tilde{\Phi}_{1}+\tilde{\Phi}_{2})$,
by Lemma \ref{lem:Phitilda2}, Lemma \ref{lem:Phitilda1}, and in
light of identities (\ref{eq:Sym1}), (\ref{eq:Sym2}) we immediately
get 
\[
\int_{\partial\mathbb{R}_{+}^{n}}qUd\bar{y}=0.
\]
Now, we use $E$ and $U$ as test functions respectively in equation
(\ref{eq:ProbBubble}) and in equation (\ref{eq:E}), obtaining
\begin{align*}
(n-2)\int_{\partial\mathbb{R}_{+}^{n}}U^{\frac{n}{n-2}}Ed\bar{y} & =\int_{\mathbb{R}_{+}^{n}}\nabla U\nabla Ed\bar{y}dy_{n}=n\int_{\partial\mathbb{R}_{+}^{n}}U^{\frac{n}{n-2}}Ed\bar{y}+\int_{\partial\mathbb{R}_{+}^{n}}qUd\bar{y}\\
 & =n\int_{\partial\mathbb{R}_{+}^{n}}U^{\frac{n}{n-2}}Ed\bar{y},
\end{align*}
thus $\int_{\partial\mathbb{R}_{+}^{n}}U^{\frac{n}{n-2}}Ed\bar{y}=0$.
At this point we can conclude the proof of the Lemma. In fact, it
is well known that the function $U$ is minimizer for 
\[
J(u)=\frac{1}{2}\int_{\mathbb{R}_{+}^{n}}|\nabla u|^{2}dy-\frac{(n-2)^{2}}{2n-2}\int|u|^{\frac{2n-2}{n-2}}dy
\]
on the Nehari manifold $\mathcal{M}:=\left\{ u\in D^{1,2}(\mathbb{R}_{+}^{n})\smallsetminus0,\ :\ \|u\|_{D^{1,2}}^{2}=(n-2)|u|_{\frac{2n-2}{n-2}}^{\frac{2n-2}{n-2}}\right\} .$
Since $E\in D^{1,2}(\mathbb{R}_{+}^{n})$ and $\int_{\partial\mathbb{R}_{+}^{n}}U^{\frac{n}{n-2}}Ed\bar{y}=0$
we have that $E\in T_{U}\mathcal{M}$ and we can compute
\[
0\le\left.\frac{d^{2}}{dt^{2}}J(U+tE)\right|_{t=0}=\int_{\mathbb{R}_{+}^{n}}|\nabla E|^{2}d\bar{y}dy_{n}-n\int_{\partial\mathbb{R}_{+}^{n}}U^{\frac{2}{n-2}}E^{2}d\bar{y}
\]
which ends the proof.
\end{proof}
We can further simplify the estimate for $-\int_{\mathbb{R}_{+}^{n}}\gamma\Delta\gamma d\bar{y}dy_{n}.$
\begin{lem}
\label{lem:Phi1}If $n>6$ we have 
\[
-\int_{\mathbb{R}_{+}^{n}}\gamma\Delta\gamma d\bar{y}dy_{n}\ge\int_{\partial\mathbb{R}_{+}^{n}}\frac{\partial\tilde{\Phi}_{1}}{\partial y_{n}}\tilde{\Phi}_{1}d\bar{y}+\int_{\partial\mathbb{R}_{+}^{n}}nU^{\frac{2}{n-2}}\tilde{\Phi}_{1}^{2}d\bar{y}-\int_{\mathbb{R}_{+}^{n}}\tilde{\Phi}_{1}\Delta\tilde{\Phi}_{1}d\bar{y}dy.
\]
\end{lem}
\begin{proof}
Combining Lemma \ref{lem:gammaDgamma} and Lemma \ref{lem:qE} we
have that 
\begin{align*}
-\int_{\mathbb{R}_{+}^{n}}\gamma\Delta\gamma d\bar{y}dy_{n} & \ge\int_{\partial\mathbb{R}_{+}^{n}}q\Phi d\bar{y}-\int_{\mathbb{R}_{+}^{n}}\Phi\Delta\Phi d\bar{y}dy\\
 & =\int_{\partial\mathbb{R}_{+}^{n}}\left(\frac{\partial\Phi}{\partial y_{n}}\Phi+nU^{\frac{2}{n-2}}\Phi^{2}\right)d\bar{y}-\int_{\mathbb{R}_{+}^{n}}\Phi\Delta\Phi d\bar{y}dy.
\end{align*}
At this point we can prove immediately by (\ref{eq:Sym3}) that 
\[
\int_{\partial\mathbb{R}_{+}^{n}}\frac{\partial\tilde{\Phi}_{1}}{\partial y_{n}}\tilde{\Phi}_{2}d\bar{y}=\int_{\partial\mathbb{R}_{+}^{n}}\frac{\partial\tilde{\Phi}_{2}}{\partial y_{n}}\tilde{\Phi}_{1}d\bar{y}=\int_{\partial\mathbb{R}_{+}^{n}}nU^{\frac{2}{n-2}}\tilde{\Phi}_{1}\tilde{\Phi}_{2}d\bar{y}=0
\]
and by (\ref{eq:Sym4}) that 
\[
\int_{\partial\mathbb{R}_{+}^{n}}\frac{\partial\tilde{\Phi}_{2}}{\partial y_{n}}\tilde{\Phi}_{2}d\bar{y}=\int_{\partial\mathbb{R}_{+}^{n}}nU^{\frac{2}{n-2}}\tilde{\Phi}_{2}^{2}d\bar{y}=0.
\]
Now, taking in account equation (\ref{eq:Phidef2}), we have
\begin{multline*}
-\int_{\mathbb{R}_{+}^{n}}\Phi\Delta\tilde{\Phi}_{2}d\bar{y}dy=\frac{1}{3}\int_{\mathbb{R}_{+}^{n}}\Phi\bar{R}_{ijkl}y_{k}y_{l}\partial_{ij}^{2}U\\
=\frac{n(n-2)}{3}\int_{\mathbb{R}_{+}^{n}}\Phi\bar{R}_{ijkl}y_{k}y_{l}y_{i}y_{j}(|\bar{y}|^{2}+(1+y_{n})^{2})^{-\frac{n}{2}}=0
\end{multline*}
again by (\ref{eq:Sym3}) and (\ref{eq:Sym4}). Similarly we prove
that $-\int_{\mathbb{R}_{+}^{n}}\tilde{\Phi}_{2}\Delta\tilde{\Phi}_{1}d\bar{y}dy=0$
and we conclude the proof.
\end{proof}

\subsection{Case $n=7$}

In this case we can take $a_{1}=a_{1}'=a_{2}'=0$ in the expression
of $\tilde{\Phi}_{1}$ given in Lemma \ref{lem:Phitilda1}, so we
set
\[
\tilde{\Phi}_{1}=R_{ninj}y_{i}y_{j}A(|\bar{y}|,y_{n}),
\]
where
\[
A(|\bar{y}|,y_{n}):=\frac{1}{12(|\bar{y}|^{2}+(1+y_{n})^{2})^{\frac{n-2}{2}}}+\frac{n-2}{6}\frac{1+y_{n}^{2}-y_{n}}{(|\bar{y}|^{2}+(1+y_{n})^{2})^{\frac{n}{2}}}
\]
and we have the final result of this subsection
\begin{lem}
\label{lem:stimafinalegamma}If $n\ge7$ we have 
\begin{multline}
-\int_{\mathbb{R}_{+}^{n}}\gamma\Delta\gamma d\bar{y}dy_{n}\ge\frac{2}{n^{2}-1}R_{ninj}^{2}\left[\int_{\partial\mathbb{R}_{+}^{n}}A(|\bar{y}|,0)\left.\frac{\partial}{\partial y_{n}}A(|\bar{y}|,y_{n})\right|_{y_{n}=0}|\bar{y}|^{4}d\bar{y}\right.\\
+n\int_{\partial\mathbb{R}_{+}^{n}}\frac{A(|\bar{y}|,0)^{2}}{|\bar{y}|^{2}+1}|\bar{y}|^{4}d\bar{y}\left.+n(n-2)\int_{\mathbb{R}_{+}^{n}}\frac{A(|\bar{y}|,y_{n})}{(|\bar{y}|^{2}+(1+y_{n})^{2})^{\frac{n+2}{2}}}|\bar{y}|^{4}y_{n}^{2}d\bar{y}dy\right]\label{eq:gDg}
\end{multline}
\end{lem}
in addition for $n=7$ 
\[
-\int_{\mathbb{R}_{+}^{n}}\gamma\Delta\gamma dy\ge\frac{29}{432}\omega_{5}I_{7}^{9}R_{7i7j}^{2}.
\]

\begin{proof}
We have, by (\ref{eq:Sym5})
\begin{align*}
\int_{\partial\mathbb{R}_{+}^{n}}\frac{\partial\tilde{\Phi}_{1}}{\partial y_{n}}\tilde{\Phi}_{1}d\bar{y} & =\int_{\partial\mathbb{R}_{+}^{n}}R_{ninj}y_{i}y_{j}A(|\bar{y}|,0)R_{nlnk}y_{l}y_{k}\left.\frac{\partial}{\partial y_{n}}A(|\bar{y}|,y_{n})\right|_{y_{n}=0}d\bar{y}\\
 & =\frac{2}{n^{2}-1}R_{ninj}^{2}\int_{\partial\mathbb{R}_{+}^{n}}A(|\bar{y}|,0)\left.\frac{\partial}{\partial y_{n}}A(|\bar{y}|,y_{n})\right|_{y_{n}=0}|\bar{y}|^{4}d\bar{y}.
\end{align*}
Similarly we have
\begin{align*}
n\int_{\partial\mathbb{R}_{+}^{n}}U^{\frac{2}{n-2}}\tilde{\Phi}_{1}^{2}d\bar{y} & =n\int_{\partial\mathbb{R}_{+}^{n}}\frac{A(|\bar{y}|,0)^{2}}{|\bar{y}|^{2}+1}R_{ninj}y_{i}y_{j}R_{nlnk}y_{l}y_{k}d\bar{y}\\
 & =\frac{2n}{n^{2}-1}R_{ninj}^{2}\int_{\partial\mathbb{R}_{+}^{n}}\frac{A(|\bar{y}|,0)^{2}}{|\bar{y}|^{2}+1}|\bar{y}|^{4}d\bar{y}.
\end{align*}
Finally, using (\ref{eq:Phidef1}) and (\ref{eq:Sym5}) we have 
\begin{align*}
-\int_{\mathbb{R}_{+}^{n}}\tilde{\Phi}_{1}\Delta\tilde{\Phi}_{1}d\bar{y}dy & =\int_{\mathbb{R}_{+}^{n}}A(|\bar{y}|,y_{n})R_{ninj}y_{i}y_{j}R_{nknl}y_{n}^{2}\partial_{kl}^{2}Ud\bar{y}dy\\
 & =R_{ninj}R_{nknl}\int_{\mathbb{R}_{+}^{n}}A(|\bar{y}|,y_{n})y_{i}y_{j}y_{n}^{2}\partial_{kl}^{2}Ud\bar{y}dy\\
 & =n(n-2)R_{ninj}R_{nknl}\int_{\mathbb{R}_{+}^{n}}\frac{A(|\bar{y}|,y_{n})}{(|\bar{y}|^{2}+(1+y_{n})^{2})^{\frac{n+2}{2}}}y_{i}y_{j}y_{k}y_{l}y_{n}^{2}d\bar{y}dy\\
 & =\frac{2n(n-2)}{n^{2}-1}R_{ninj}^{2}\int_{\mathbb{R}_{+}^{n}}\frac{A(|\bar{y}|,y_{n})}{(|\bar{y}|^{2}+(1+y_{n})^{2})^{\frac{n+2}{2}}}|\bar{y}|^{4}y_{n}^{2}d\bar{y}dy
\end{align*}
which proves the first claim. 

To conclude the proof we will have to estimate several integral quantities
involving the functions $A(|\bar{y}|,y_{n})$ and its derivative
\[
\left.\frac{\partial}{\partial y_{n}}A(|\bar{y}|,y_{n})\right|_{y_{n}=0}=-\frac{5}{4}\left(1+|\bar{y}|^{2}\right)^{-\frac{7}{2}}-\frac{35}{6}\left(1+|\bar{y}|^{2}\right)^{-\frac{9}{2}}
\]
which we compute below. Notice also that, by change of variables,
we have 
\[
\int_{\partial\mathbb{R}_{+}^{7}}\frac{|\bar{y}|^{4}d\bar{y}}{\left(1+|\bar{y}|^{2}\right)^{\alpha}}=\omega_{5}I_{\alpha}^{9}\text{ and}\int_{\mathbb{R}_{+}^{7}}\frac{|\bar{y}|^{4}y_{n}^{\beta}d\bar{y}dy_{n}}{\left((1+y_{7})^{2}+|\bar{y}|^{2}\right)^{\alpha}}=\omega_{5}I_{\alpha}^{9}\int_{0}^{\infty}\frac{t^{\beta}dt}{(1+t)^{2\alpha-10}}.
\]
Keeping in mind (\ref{eq:Iam}) we have 

\begin{equation}
\int_{\partial\mathbb{R}_{+}^{7}}A\frac{\partial}{\partial y_{n}}A|\bar{y}|^{4}d\bar{y}=-\omega_{5}\frac{85}{24}I_{7}^{9}.\label{eq:AdA-1}
\end{equation}
and
\begin{equation}
7\int_{\partial\mathbb{R}_{+}^{7}}\frac{A^{2}|\bar{y}|^{4}}{\left(1+|\bar{y}|^{2}\right)}d\bar{y}=\omega_{5}\frac{191}{72}I_{7}^{9}.\label{eq:A2-1}
\end{equation}
Finally, in light of (\ref{eq:t-integrali}), we have
\begin{equation}
35\int_{\mathbb{R}_{+}^{7}}\frac{A|\bar{y}|^{4}y_{n}^{2}}{\left((1+y_{n}^{2})+|\bar{y}|^{2}\right)^{\frac{9}{2}}}d\bar{y}dy_{n}=\omega_{5}\frac{5}{2}I_{7}^{9}.\label{eq:35A-1}
\end{equation}
By (\ref{eq:AdA-1}), (\ref{eq:A2-1}) and (\ref{eq:35A-1}) we get
the proof
\end{proof}
\begin{proof}[Proof of first claim of Lemma \ref{lem:poho7}]
 By (\ref{eq:poho2}), (\ref{eq:Iam}), and by Lemma \ref{lem:stimafinalegamma}
we immediately get (\ref{eq:pohofinale7}).
\end{proof}

\subsection{Case $n=8$}

For $n=8$ we want to repeat the same strategy used for $n=7$. Unfortunately,
taking all the coefficients equal to zero in $\tilde{\Phi}_{1}$ does
not prove the sign condition. For this case thus we consider 
\[
\tilde{\Phi}_{1}=R_{ninj}y_{i}y_{j}A(|\bar{y}|,y_{n},b),
\]
where
\[
A(|\bar{y}|,y_{8},b):=\frac{1}{12(|\bar{y}|^{2}+(1+y_{8})^{2})^{3}}+\frac{1+y_{n}^{2}-y_{n}}{(|\bar{y}|^{2}+(1+y_{8})^{2})^{4}}+\frac{b}{(|\bar{y}|^{2}+(1+y_{8})^{2})^{5}}.
\]

\begin{lem}
\label{lem:stimafinalegamma8}For $n=8$ and $b=-2$ we have 
\[
-\int_{\mathbb{R}_{+}^{8}}\gamma\Delta\gamma dy\ge\frac{121}{13601}\omega_{6}I_{8}^{10}R_{8i8j}^{2}
\]
\end{lem}
\begin{proof}
We can recast (\ref{eq:gDg}) for $n=8$ obtaining
\begin{multline}
-\int_{\mathbb{R}_{+}^{8}}\gamma\Delta\gamma d\bar{y}dy_{8}\ge\frac{2}{63}R_{8i8j}^{2}\left[\int_{\partial\mathbb{R}_{+}^{8}}A(|\bar{y}|,0,b)\left.\frac{\partial}{\partial y_{8}}A(|\bar{y}|,y_{8},b)\right|_{y_{8}=0}|\bar{y}|^{4}d\bar{y}\right.\\
+8\int_{\partial\mathbb{R}_{+}^{8}}\frac{A(|\bar{y}|,0,b)^{2}|\bar{y}|^{4}}{|\bar{y}|^{2}+1}d\bar{y}\left.+48\int_{\mathbb{R}_{+}^{8}}\frac{A(|\bar{y}|,y_{8},b)|\bar{y}|^{4}y_{8}^{2}}{(|\bar{y}|^{2}+(1+y_{8})^{2})^{5}}d\bar{y}dy\right].\label{eq:stimafinal8}
\end{multline}
We have 
\begin{equation}
\frac{1}{\omega_{6}}\int_{\partial\mathbb{R}_{+}^{8}}\left.A\frac{\partial A}{\partial y_{8}}\right|_{y_{8}=0}|\bar{y}|^{4}d\bar{y}=\left[-\frac{21}{4}-\frac{35}{12}b-\frac{35}{64}b^{2}\right]I_{8}^{10},\label{eq:AdA8}
\end{equation}
\begin{equation}
\frac{8}{\omega_{6}}\int_{\partial\mathbb{R}_{+}^{8}}\frac{A^{2}|\bar{y}|^{4}}{|\bar{y}|^{2}+1}d\bar{y}=I_{8}^{10}\left[\frac{221}{54}+\frac{85}{36}b+\frac{7}{16}b^{2}\right]\label{eq:AA8}
\end{equation}
and
\begin{equation}
\frac{48}{\omega_{6}}\int_{\mathbb{R}_{+}^{8}}\frac{A(|\bar{y}|,y_{8},b)|\bar{y}|^{4}y_{8}^{2}}{(|\bar{y}|^{2}+(1+y_{8})^{2})^{5}}d\bar{y}dy=I_{8}^{10}\left[\frac{5}{6}+b\frac{5}{144}\right].\label{eq:finale8}
\end{equation}
So by (\ref{eq:AdA8}), (\ref{eq:AA8}) and (\ref{eq:finale8}), the
inequality (\ref{eq:stimafinal8}) becomes
\[
-\int_{\mathbb{R}_{+}^{8}}\gamma\Delta\gamma d\bar{y}dy_{8}\ge\frac{2}{63}R_{8i8j}^{2}\omega_{6}I_{8}^{10}\left[-\frac{35}{108}-\frac{25}{48}b-\frac{7}{64}b^{2}\right]
\]
which for $b=-2$ gives the claim.
\end{proof}
\begin{proof}[Proof of second claim of Lemma \ref{lem:poho7}]
 By (\ref{eq:poho2}), (\ref{eq:Iam}), and by Lemma \ref{lem:stimafinalegamma8}
we immediately get (\ref{eq:pohofinale8}).
\end{proof}

\section{\label{sec:6}Case $n=6$}

When dealing with low dimensions, often it is convenient to work in
cylindrical sets 
\[
D_{r}^{+}:=[0,r]\times B_{r}^{5}\subset\mathbb{R}_{+}^{6}
\]
instead of spheres $B_{r}^{+}=B_{r}^{6}\cap\mathbb{R}_{+}^{6}$. In
the limit $r\rightarrow\infty$ the difference between the two approaches
is of higher order, but the boundary of $D_{r}^{+}$ is easier to
manage. So, we compute the Pohozaev identity on cylindrical sets.
Again, as in \cite[Proposition 14]{GM} we have that, if $x_{i}\rightarrow x_{0}$
is isolated simple blow-up point for $u_{i}$, then
\begin{equation}
P(u_{i},r)\ge R(U,U)+R(U,\delta_{i}^{2}\gamma_{})+R(\delta_{i}^{2}\gamma_{},U)+O(\delta_{i}^{4})\label{eq:poho1-1}
\end{equation}
where $R(u,v)$ in this case is
\[
R(u,v):=-\int_{D_{r/\delta}^{+}}\left(y^{b}\partial_{b}u+\frac{n-2}{2}u\right)\left[(L_{\hat{g}_{i}}-\Delta)v\right]dy.
\]
Throughout this section we will the following lemma.
\begin{lem}
\label{lem:poho6}Let $x_{i}\rightarrow x_{0}$ is an isolated simple
blow-up point for $u_{i}$ solution of (\ref{eq:Prob-i}) then it
holds
\begin{align}
P(u_{i},r) & \ge R(U,U)+R(U,\delta_{i}^{2}\gamma_{x_{i}})+R(\delta_{i}^{2}\gamma_{x_{i}},U)+O(\delta_{i}^{4})\nonumber \\
 & =\omega_{4}I_{6}^{6}\delta_{i}^{4}\log\left(\frac{1}{\delta_{i}}\right)\left[\frac{8}{45}|\bar{W}(x_{i})|^{2}+\frac{8}{15}R_{6i6s}^{2}(x_{i})\right]+O(\delta_{i}^{4})\label{eq:pohofinale6}
\end{align}
\end{lem}
\begin{rem}
We recall the following elementary identity, obtained by change of
variables
\begin{equation}
\int_{0}^{r}\int_{B_{r}^{n-1}}\frac{|\bar{y}|^{\beta}dy}{\left[(1+y_{n})^{2}+|\bar{y}|^{2}\right]^{\alpha}}=\int_{0}^{r}\frac{(1+y_{n})^{\beta+n-1}dy_{n}}{(1+y_{n})^{2\alpha}}\int_{B_{r}^{n-1}}\frac{|\bar{y}|^{\beta}d\bar{y}}{\left[1+|\bar{y}|^{2}\right]^{\alpha}}\label{eq:idelem}
\end{equation}
and, finally, that 
\[
\int_{0}^{r/\delta}\frac{y_{n}^{\alpha}}{(1+y_{n})^{\alpha+1}}=\log\left(\frac{1}{\delta}\right)+O(1),\text{ and \ \ensuremath{\int_{0}^{r/\delta}}}\frac{y_{n}^{\alpha+2}-y_{n}^{\alpha}}{(1+y_{n})^{\alpha+3}}=\log\left(\frac{1}{\delta}\right)+O(1)
\]
for $\alpha=0,2,4$
\end{rem}
With these premises we have the following result (in order to simplify
notation, we denote $\delta$ for $\delta_{i}$ and $q$ for $x_{i}$).
\begin{lem}
\label{lem:R(UU)}We have 
\[
R(U,U)=\omega_{4}I_{6}^{6}\delta^{4}\log\left(\frac{r}{\delta}\right)\left[\frac{8}{45}|\bar{W}(q)|^{2}-\frac{16}{15}R_{nins}^{2}\right]+O(\delta^{4}).
\]
\end{lem}
\begin{proof}
The proof is similar to \cite[Lemma 15]{GM}. We focus here on the
main differences, omitting the standard calculations. By definition
of $L_{\hat{g}_{i}}$ we have
\begin{align*}
R(U,U) & =\frac{\left(n-2\right)^{2}}{2}\int_{D_{r\delta^{-1}}^{+}}\frac{|y|^{2}-1}{\left[(1+y_{n})^{2}+|\bar{y}|^{2}\right]^{n+1}}ny_{i}y_{j}\left(g^{ij}(\delta y)-\delta^{ij}\right)dy\\
 & -\frac{\left(n-2\right)^{2}}{2}\int_{D_{r\delta^{-1}}^{+}}\frac{|y|^{2}-1}{\left[(1+y_{n})^{2}+|\bar{y}|^{2}\right]^{n}}\left(g^{jj}(\delta y)-1\right)dy\\
 & -\frac{\left(n-2\right)^{2}}{2}\int_{D_{r\delta^{-1}}^{+}}\frac{|y|^{2}-1}{\left[(1+y_{n})^{2}+|\bar{y}|^{2}\right]^{n}}\delta\partial_{i}g^{ij}(\delta y)y_{j}dy\\
 & -\frac{\left(n-2\right)^{2}}{8(n-1)}\int_{D_{r\delta^{-1}}^{+}}\frac{|y|^{2}-1}{\left[(1+y_{n})^{2}+|\bar{y}|^{2}\right]^{n-1}}\delta^{2}R_{g}(\delta y)dy+O(\delta^{4})\\
 & =:A_{1}+A_{2}+A_{3}+A_{4}+O(\delta^{4}).
\end{align*}
Using the symmetries of the curvature tensor and the expansion of
the metric we have that, for $n=6$, 
\begin{align}
A_{1}= & \delta^{4}\frac{24}{5}R_{nins}^{2}\int_{D_{r\delta^{-1}}^{+}}\frac{(|y|^{2}-1)|\bar{y}|^{2}y_{n}^{4}}{\left[(1+y_{n})^{2}+|\bar{y}|^{2}\right]^{7}}dy\label{eq:A1}\\
 & +\delta^{4}\frac{48}{35}R_{ninj,ji}\int_{D_{r\delta^{-1}}^{+}}\frac{(|y|^{2}-1)|\bar{y}|^{4}y_{n}^{2}}{\left[(1+y_{n})^{2}+|\bar{y}|^{2}\right]^{7}}dy+O(\delta^{4}).\nonumber 
\end{align}
\begin{align}
A_{2}+A_{3} & =-\delta^{4}4R_{nins}^{2}\int_{D_{r\delta^{-1}}^{+}}\frac{(|y|^{2}-1)y_{n}^{4}}{\left[(1+y_{n})^{2}+|\bar{y}|^{2}\right]^{6}}dy\label{eq:A2+A3}\\
 & -\delta^{4}\frac{8}{5}R_{ninj,ij}\int_{D_{r\delta^{-1}}^{+}}\frac{(|y|^{2}-1)|\bar{y}|^{2}y_{n}^{2}}{\left[(1+y_{n})^{2}+|\bar{y}|^{2}\right]^{6}}dy+O(\delta^{4}).\nonumber 
\end{align}
and 
\begin{align}
A_{4} & =\delta^{4}\frac{1}{150}|\bar{W}(q)|^{2}\int_{D_{r\delta^{-1}}^{+}}\frac{(|y|^{2}-1)|\bar{y}|^{2}}{\left[(1+y_{n})^{2}+|\bar{y}|^{2}\right]^{5}}dy\label{eq:A4}\\
 & +\delta^{4}\frac{2}{5}R_{nins}^{2}\int_{D_{r\delta^{-1}}^{+}}\frac{(|y|^{2}-1)y_{n}^{2}}{\left[(1+y_{n})^{2}+|\bar{y}|^{2}\right]^{5}}dy\nonumber \\
 & +\delta^{4}\frac{2}{5}R_{ninj,ij}\int_{D_{r\delta^{-1}}^{+}}\frac{(|y|^{2}-1)y_{n}^{2}}{\left[(1+y_{n})^{2}+|\bar{y}|^{2}\right]^{5}}dy+O(\delta^{4}).\nonumber 
\end{align}
Now, using (\ref{eq:idelem}) we have
\begin{multline*}
\int_{D_{r\delta^{-1}}^{+}}\frac{(|y|^{2}-1)|\bar{y}|^{2}y_{n}^{4}}{\left[(1+y_{n})^{2}+|\bar{y}|^{2}\right]^{7}}\\
=\int_{0}^{r/\delta}y_{n}^{4}dy_{n}\int_{B_{r/\delta}^{5}}\frac{|\bar{y}|^{4}d\bar{y}}{\left[(1+y_{n})^{2}+|\bar{y}|^{2}\right]^{7}}+\int_{0}^{r/\delta}y_{n}^{4}(y_{n}^{2}-1)dy_{n}\int_{B_{r/\delta}^{5}}\frac{|\bar{y}|^{2}d\bar{y}}{\left[(1+y_{n})^{2}+|\bar{y}|^{2}\right]^{7}}\\
=\int_{0}^{r/\delta}\frac{y_{n}^{4}}{(1+y_{n})^{5}}dy_{n}\int_{B_{r/\delta}^{5}}\frac{|\bar{y}|^{4}d\bar{y}}{\left[1+|\bar{y}|^{2}\right]^{7}}+\int_{0}^{r/\delta}\frac{y_{n}^{6}-y_{n}^{4}}{(1+y_{n})^{7}}dy_{n}\int_{B_{r/\delta}^{5}}\frac{|\bar{y}|^{2}d\bar{y}}{\left[1+|\bar{y}|^{2}\right]^{7}}\\
=\omega_{4}\int_{0}^{r/\delta}\frac{y_{n}^{4}}{(1+y_{n})^{5}}dy_{n}\int_{0}^{r/\delta}\frac{\rho^{8}d\bar{y}}{\left[1+\rho^{2}\right]^{7}}+\omega_{4}\int_{0}^{r/\delta}\frac{y_{n}^{6}-y_{n}^{4}}{(1+y_{n})^{7}}dy_{n}\int_{B_{r/\delta}^{5}}\frac{\rho^{6}d\bar{y}}{\left[1+\rho^{2}\right]^{7}}\\
=\log(1/\delta))(I_{7}^{8}+I_{7}^{6})+O(1).
\end{multline*}
In a similar way we proceed for all the terms in (\ref{eq:A1}), (\ref{eq:A2+A3}),
(\ref{eq:A4}), obtaining
\begin{align}
A_{1}= & \omega_{4}\delta^{4}\log\left(\frac{1}{\delta}\right)\left[\frac{24}{5}R_{nins}^{2}(I_{7}^{8}+I_{7}^{6})+\frac{48}{35}R_{ninj,ji}(I_{7}^{10}+I_{7}^{8})\right]+O(\delta^{4}),\label{eq:A1-1}
\end{align}
\begin{align}
A_{2}+A_{3} & =-\omega_{4}\delta^{4}\log\left(\frac{1}{\delta}\right)\left[4R_{nins}^{2}(I_{6}^{6}+I_{6}^{4})+\frac{8}{5}R_{ninj,ij}(I_{6}^{8}+I_{6}^{6})\right]+O(\delta^{4}),\label{eq:A2+A3-1}
\end{align}
and 
\begin{align}
A_{4} & =\omega_{4}\delta^{4}\log\left(\frac{1}{\delta}\right)\frac{1}{150}|\bar{W}(q)|^{2}(I_{5}^{8}+I_{5}^{6})\label{eq:A4-1}\\
 & +\omega_{4}\delta^{4}\log\left(\frac{1}{\delta}\right)\frac{2}{5}\left(R_{nins}^{2}+R_{ninj,ij}\right)(I_{5}^{6}+I_{5}^{4}).\nonumber 
\end{align}
Finally, in light of (\ref{eq:Iam}), by (\ref{eq:A1-1}), (\ref{eq:A2+A3-1})
and (\ref{eq:A4-1}) we get the claim.
\end{proof}
\begin{lem}
\label{lem:gammaN5}We have
\[
R(U,\delta^{2}\gamma)+R(\delta^{2}\gamma,U)=-2\delta^{4}\int_{D_{r\delta^{-1}}^{+}}\gamma\Delta\gamma dy+O(\delta^{4}).
\]
\end{lem}
\begin{proof}
Again, we follow the main lines of \cite[Lemma 16]{GM}. We have by
definition of $R(u,v)$, by (\ref{eq:gij}) and (\ref{eq:gradvq}),
that
\begin{align*}
R(U,\delta^{2}\gamma)+R(\delta^{2}\gamma,U)= & -\delta^{4}\int_{D_{r\delta^{-1}}^{+}}\left(y_{b}\partial_{b}U+2U\right)\left[\frac{1}{3}\bar{R}_{ikjl}y_{k}y_{l}+R_{ninj}y_{n}^{2}\right]\partial_{i}\partial_{j}\gamma dy\\
 & -\delta^{4}\int_{D_{r\delta^{-1}}^{+}}y_{b}\partial_{b}\gamma\left[\frac{1}{3}\bar{R}_{ikjl}y_{k}y_{l}+R_{ninj}y_{n}^{2}\right]\partial_{i}\partial_{j}Udy\\
 & -\delta^{4}\int_{D_{r\delta^{-1}}^{+}}2\gamma\left[\frac{1}{3}\bar{R}_{ikjl}y_{k}y_{l}+R_{ninj}y_{n}^{2}\right]\partial_{i}\partial_{j}Udy+O(\delta^{4})\\
=:- & \delta^{4}(A_{1}+A_{2}+A_{3})+O(\delta^{4}).
\end{align*}
By (\ref{eq:vqdef}), immediately we have
\begin{equation}
A_{3}=2\int_{D_{r\delta^{-1}}^{+}}\gamma\Delta\gamma.\label{eq:A3-1}
\end{equation}
Integrating by parts, and recalling that the index \textbf{$b=1,\dots,n$}
while $i,j,k,l,s=1,\dots,n-1$, we have 
\begin{align}
A_{2}= & 6\int_{D_{r\delta^{-1}}^{+}}\gamma\left[\frac{1}{3}\bar{R}_{ikjl}y_{k}y_{l}+R_{ninj}y_{n}^{2}\right]\partial_{i}\partial_{j}Udy\nonumber \\
 & +\int_{D_{r\delta^{-1}}^{+}}y_{b}\gamma\partial_{b}\left[\frac{1}{3}\bar{R}_{ikjl}y_{k}y_{l}+R_{ninj}y_{n}^{2}\right]\partial_{i}\partial_{j}Udy\nonumber \\
 & +\int_{D_{r\delta^{-1}}^{+}}y_{b}\gamma\left[\frac{1}{3}\bar{R}_{ikjl}y_{k}y_{l}+R_{ninj}y_{n}^{2}\right]\partial_{b}\partial_{i}\partial_{j}Udy\nonumber \\
 & -\int_{0}^{r/\delta}\int_{\partial B_{r/\delta}^{5}}y_{s}\nu_{s}\gamma\left[\frac{1}{3}\bar{R}_{ikjl}y_{k}y_{l}+R_{ninj}y_{n}^{2}\right]\partial_{i}\partial_{j}Ud\sigma dy_{n}\nonumber \\
 & -\int_{B_{r/\delta}^{5}}\left.y_{n}\gamma\left[\frac{1}{3}\bar{R}_{ikjl}y_{k}y_{l}+R_{ninj}y_{n}^{2}\right]\partial_{i}\partial_{j}U\right|_{y_{n}=\frac{r}{\delta}}d\bar{y}.\label{eq:A2intermedio}
\end{align}
Now, we estimate the boundary terms. On $\partial B_{r/\delta}^{5}$
we have $y_{s}\nu_{s}=|\bar{y}|=r/\delta$. Taking in account (\ref{eq:gradvq})
we get
\begin{multline*}
\left|\int_{0}^{r/\delta}\int_{\partial B_{r/\delta}^{5}}y_{s}\nu_{s}\gamma\left[\frac{1}{3}\bar{R}_{ikjl}y_{k}y_{l}+R_{ninj}y_{n}^{2}\right]\partial_{i}\partial_{j}Ud\sigma dy_{n}\right|\\
\le C\int_{0}^{r/\delta}\int_{\partial B_{r/\delta}^{5}}\left(\frac{r}{\delta}\right)^{-5}d\sigma dy_{n}=O(1).
\end{multline*}
Similarly we obtain $\left|\int_{B_{r/\delta}^{5}}\left.y_{n}\gamma_{q}\left[\frac{1}{3}\bar{R}_{ikjl}y_{k}y_{l}+R_{ninj}y_{n}^{2}\right]\partial_{i}\partial_{j}U\right|_{y_{n}=\frac{r}{\delta}}d\bar{y}\right|=O(1)$. 

Moreover, 
\begin{multline*}
\int_{D_{r\delta^{-1}}^{+}}y_{b}\gamma\partial_{b}\left[\frac{1}{3}\bar{R}_{ikjl}y_{k}y_{l}+R_{ninj}y_{n}^{2}\right]\partial_{i}\partial_{j}Udy\\
=\int_{D_{r\delta^{-1}}^{+}}y_{s}\gamma\partial_{s}\left[\frac{1}{3}\bar{R}_{ikjl}y_{k}y_{l}\right]\partial_{i}\partial_{j}Udy+\int_{\mathbb{R}_{+}^{n}}y_{n}\gamma\partial_{n}\left[R_{ninj}y_{n}^{2}\right]\partial_{i}\partial_{j}Udy\\
=2\int_{D_{r\delta^{-1}}^{+}}\gamma\left[\frac{1}{3}\bar{R}_{ikjl}y_{k}y_{l}+R_{ninj}y_{n}^{2}\right]\partial_{i}\partial_{j}Udy=-2\int_{D_{r\delta^{-1}}^{+}}\gamma\Delta\gamma,
\end{multline*}
and, using (\ref{eq:vqdef}) for the first term of (\ref{eq:A2intermedio})
we have
\[
A_{2}=-8\int_{D_{r\delta^{-1}}^{+}}\gamma\Delta\gamma dy+\int_{D_{r\delta^{-1}}^{+}}y_{b}\gamma\left[\frac{1}{3}\bar{R}_{ikjl}y_{k}y_{l}+R_{ninj}y_{n}^{2}\right]\partial_{b}\partial_{i}\partial_{j}Udy+O(1).
\]
For the term $A_{1}$ we integrate by parts twice. As before, all
the boundary terms are estimated by a constant number. So, using the
symmetries of the curvature tensor we have, after the first integration,
\begin{align*}
A_{1}= & \int_{D_{r\delta^{-1}}^{+}}\left(\partial_{i}U+y_{b}\partial_{i}\partial_{b}U+2\partial_{i}U\right)\left[\frac{1}{3}\bar{R}_{ikjl}y_{k}y_{l}+R_{ninj}y_{n}^{2}\right]\partial_{j}\gamma dy\\
 & +\int_{D_{r\delta^{-1}}^{+}}\left(y_{b}\partial_{b}U+2U\right)\partial_{i}\left[\frac{1}{3}\bar{R}_{ikjl}y_{k}y_{l}+R_{ninj}y_{n}^{2}\right]\partial_{j}\gamma dy+O(1)\\
 & =\int_{D_{r\delta^{-1}}^{+}}\left(3\partial_{i}U+y_{b}\partial_{i}\partial_{b}U\right)\left[\frac{1}{3}\bar{R}_{ikjl}y_{k}y_{l}+R_{ninj}y_{n}^{2}\right]\partial_{j}\gamma dy+O(1)
\end{align*}
And, integrating again, 
\begin{align*}
A_{1}= & -\int_{D_{r\delta^{-1}}^{+}}\left(4\partial_{j}\partial_{i}U+y_{b}\partial_{j}\partial_{i}\partial_{b}U\right)\left[\frac{1}{3}\bar{R}_{ikjl}y_{k}y_{l}+R_{ninj}y_{n}^{2}\right]\gamma dy+O(1)\\
= & 4\int_{D_{r\delta^{-1}}^{+}}\gamma\Delta\gamma dy-\int_{D_{r\delta^{-1}}^{+}}y_{b}\partial_{j}\partial_{i}\partial_{b}U\left[\frac{1}{3}\bar{R}_{ikjl}y_{k}y_{l}+R_{ninj}y_{n}^{2}\right]\gamma dy+O(1).
\end{align*}
Adding $A_{1}$, $A_{2}$ and $A_{3}$ we get the proof.
\end{proof}
\begin{lem}
\label{lem:Phi-n5}We have 
\[
\int_{D_{r\delta^{-1}}^{+}}\gamma\Delta\gamma dy=\int_{D_{r\delta^{-1}}^{+}}\tilde{\Phi}_{1}\Delta\tilde{\Phi}_{1}dy+O(1).
\]
\end{lem}
\begin{proof}
We have, integrating by parts, and since $E$ is harmonic
\begin{multline}
\int_{D_{r\delta^{-1}}^{+}}\gamma\Delta\gamma dy=\int_{D_{r\delta^{-1}}^{+}}(\Phi+E)\Delta(\Phi+E)dy=\int_{D_{r\delta^{-1}}^{+}}(\Phi+E)\Delta\Phi dy\\
=-\int_{D_{r\delta^{-1}}^{+}}\nabla(\Phi+E)\nabla\Phi dy+\int_{0}^{r/\delta}\int_{\partial B_{r/\delta}^{5}}(\Phi+E)\nabla\Phi\cdot\nu d\sigma dy_{n}.\label{eq:gammaDgamma}
\end{multline}
By the decay of $\gamma$ in (\ref{eq:gradvq}) and by the explicit
expression of $\Phi$ in Corollary \ref{cor:Phi1e2}, we obtain 
\[
|\nabla^{\tau}E(y)|\le C(1+|y|)^{4-\tau-n}\text{ for }\tau=0,1,
\]
and the same holds for $\Phi$. At this point we can easily see that
\[
\left|\int_{0}^{r/\delta}\int_{\partial B_{r/\delta}^{5}}(\Phi+E)\nabla\Phi\cdot\nu d\sigma dy_{n}\right|=O(1).
\]
We can integrate again by parts in (\ref{eq:gammaDgamma}) and, keeping
in mind that again the boundary term is estimate by a constant and
that $E$ is harmonic, we get
\begin{align*}
\int_{D_{r\delta^{-1}}^{+}}\gamma\Delta\gamma dy & =-\int_{D_{r\delta^{-1}}^{+}}\nabla(\Phi+E)\nabla\Phi dy+O(1)\\
 & =\int_{D_{r\delta^{-1}}^{+}}\Delta(\Phi+E)\Phi dy+O(1)=\int_{D_{r\delta^{-1}}^{+}}\Phi\Delta\Phi dy+O(1).
\end{align*}
Now we proceed in Lemma \ref{lem:Phi1}, using (\ref{eq:Sym3}) and
(\ref{eq:Sym4}) to prove that 
\[
\int_{D_{r\delta^{-1}}^{+}}\Phi\Delta\Phi dy=\int_{D_{r\delta^{-1}}^{+}}\tilde{\Phi}_{1}\Delta\tilde{\Phi}_{1}dy
\]
and concluding the proof.
\end{proof}
\begin{lem}
\label{lem:R(udelta)}We have
\[
R(U,\delta^{2}\gamma)+R(\delta^{2}\gamma,U)=\omega_{4}I_{6}^{6}\delta^{4}\log\left(\frac{r}{\delta}\right)\frac{24}{15}R_{nins}^{2}+O(\delta^{4}).
\]
\end{lem}
\begin{proof}
By Lemma \ref{lem:gammaN5} and Lemma \ref{lem:Phi-n5}, taking in
account (\ref{eq:Sym5}), we have
\begin{multline*}
R(U,\delta^{2}\gamma)+R(\delta^{2}\gamma,U)=-2\delta^{4}\int_{D_{r\delta^{-1}}^{+}}\tilde{\Phi}_{1}\Delta\tilde{\Phi}_{1}dy+O(\delta^{4})\\
=\frac{8}{35}\delta^{4}R_{nins}^{2}\left[\int_{D_{r\delta^{-1}}^{+}}\frac{|\bar{y}|^{4}y_{n}^{2}dy}{(|\bar{y}|^{2}+(1+y_{n})^{2})^{6}}+8\int_{D_{r\delta^{-1}}^{+}}\frac{|\bar{y}|^{4}(y_{n}^{2}+y_{n}^{4}-y_{n}^{3})dy}{(|\bar{y}|^{2}+(1+y_{n})^{2})^{7}}\right]+O(\delta^{4})\\
=:\frac{8}{35}\delta^{4}R_{nins}^{2}\left[B_{1}+B_{1}\right]+O(\delta^{4}).
\end{multline*}
By (\ref{eq:idelem}) we have
\begin{align*}
B_{1}= & \int_{0}^{r/\delta}\frac{y_{n}^{2}dy_{n}}{(1+y_{n})^{3}}\int_{B_{r/\delta}^{5}}\frac{|\bar{y}|^{4}d\bar{y}}{(1+|\bar{y}|^{2})^{6}}=\log\left(\frac{1}{\delta}\right)I_{6}^{8}\\
B_{2}= & 8\int_{0}^{r/\delta}\frac{(y_{n}^{2}+y_{n}^{4}-y_{n}^{3})dy_{n}}{(1+y_{n})^{5}}\int_{B_{r/\delta}^{5}}\frac{|\bar{y}|^{4}d\bar{y}}{(1+|\bar{y}|^{2})^{7}}=8\log\left(\frac{1}{\delta}\right)I_{7}^{8}
\end{align*}
and, by (\ref{rem:Iam}) we get the proof.
\end{proof}
\begin{proof}[Proof of Lemma \ref{lem:poho6}.]
 Lemma \ref{lem:R(UU)} and Lemma \ref{lem:R(udelta)} lead us to
(\ref{eq:pohofinale6}).
\end{proof}

\section{\label{sec:Proof}Proof of the main result}

We start proving the following Weyl vanishing property.
\begin{prop}
\label{prop:vanish}Let $n\ge6$ and let $x_{i}\rightarrow x_{0}$
be an isolated simple blow up point for $u_{i}$ solution of (\ref{eq:Prob-i})
Then
\[
W(x_{0})=0.
\]
\end{prop}
\begin{proof}
By \cite[Propositions 5 and 18]{GM} we have 
\[
P(u_{i},r)\le C\delta_{i}^{n-2}.
\]
This, combined with (\ref{eq:pohofinale7}), (\ref{eq:pohofinale8})
and with (\ref{eq:pohofinale6}) gives
\[
|\bar{W}(x_{i})|^{2}+R_{ninj}^{2}(x_{i})\le\left\{ \begin{array}{cc}
C\delta_{i}^{2} & \text{ for }n=8\\
C\delta_{i} & \text{ for }n=7\\
-C\left(\log(\delta_{i})\right)^{-1} & \text{ for }n=6
\end{array}\right.,
\]
which gives the result, since $W(x_{0})=0$ if and only if both $\bar{W}$
and $R_{nins}^{2}$ vanish at $x_{0}$.
\end{proof}
Now we give a series of results whose proofs are very similar to the
ones contained in \cite{GM}, so we will omit them.

First, we can rule out the possibility to have isolated blow up points
which are not simple. As in the previous proposition, for the proof
it is crucial that $P(u_{i},r)$ is strictly positive when $|W(x_{0})|\neq0$,
which we have proved in equations (\ref{eq:pohofinale7}) and (\ref{eq:pohofinale6}). 
\begin{prop}
\label{prop:isolato->semplice}Assume $n\ge6$. Let $x_{i}\rightarrow x_{0}$
be an isolated simple blow up point for $u_{i}$ solution of (\ref{eq:Prob-i}).
Assume $|W(x_{0})|\neq0$. Then $x_{0}$ is isolated simple.
\end{prop}
Next, we can prove a splitting lemma.
\begin{prop}
\label{prop:Splitting 1}Assume $n\ge6$. Given $\beta>0$ and $R>0$
there exist two constants $C_{0},C_{1}>0$ (depending on $\beta$,
$R$ and $(M,g)$) such that, if $u$ is a solution of

\begin{equation}
\left\{ \begin{array}{cc}
L_{g}u=0 & \text{ in }M\\
B_{g}u+(n-2)f^{-\tau_{}}u^{p}=0 & \text{ on }\partial M
\end{array}\right.\label{eq:Prob-p}
\end{equation}
 and $\max_{\partial M}u>C_{0}$, then $\tau:=\frac{n}{n-2}-p<\beta$
and there exist $q_{1},\dots,q_{N}\in\partial M$, with $N=N(u)\ge1$
with the following properties: for $j=1,\dots,N$
\begin{enumerate}
\item Set $r_{j}:=Ru(q_{j})^{1-p}$, then $\left\{ B_{r_{j}}\cap\partial M\right\} _{j}$
are a disjoint collection;
\item we have $\left|u(q_{j})^{-1}u(\psi_{j}(y))-U(u(q_{j})^{p-1}y)\right|_{C^{2}(B_{2r_{j}}^{+})}<\beta$
(here $\psi_{j}$ are the Fermi coordinates at point $q_{j}$; 
\item we have
\begin{align*}
u(x)d_{\bar{g}}\left(x,\left\{ q_{1},\dots,q_{n}\right\} \right)^{\frac{1}{p-1}}\le C_{1} & \text{ for all }x\in\partial M\\
u(q_{j})d_{\bar{g}}\left(q_{j},q_{k}\right)^{\frac{1}{p-1}}\ge C_{0} & \text{ for any }j\neq k.
\end{align*}
Here $d_{\bar{g}}$ is the geodesic distance on $\partial M$.
\end{enumerate}
Assume also $W(x)\neq0$ for any $x\in\partial M$. Then there exists
$d=d(\beta,R)$ such that, for any $u$ solution of (\ref{eq:Prob-p})
with $\max_{\partial M}u>C_{0}$, we have
\[
\min_{\begin{array}{c}
i\neq j\\
1\le i,j\le N(u)
\end{array}}d_{\bar{g}}(q_{i}(u),q_{j}(u))\ge d.
\]
 
\end{prop}
Now we can prove our main result.
\begin{proof}[Proof of Theorem \ref{thm:main}.]
 We proceed by contradiction, supposing that there exists a sequence
of solutions $\left\{ u_{i}\right\} _{i}$ of problems (\ref{eq:Prob-i})
and that $x_{i}\rightarrow x_{0}$ is a blow up point for $u_{i}$.
Let $q_{1}(u_{i}),\dots q_{N(u_{i})}(u_{i})$ the sequence of points
given by proposition \ref{prop:Splitting 1}. We can prove that $d_{\bar{g}}(x_{i},q_{k_{i}}(u_{i}))\rightarrow0$
for some sequence of $k_{i}$. So $q_{k_{i}}\rightarrow x_{0}$ is
a blow up point for $u_{i}$. Now by propositions \ref{prop:Splitting 1}
and \ref{prop:isolato->semplice} we have that $q_{k_{i}}\rightarrow x_{0}$
is an isolated simple blow up point for $u_{i}$. Then, by \ref{prop:vanish},
this should imply that $|W(x_{0})|=0$ which contradicts our hypotheses,
proving the theorem.
\end{proof}

\section{Appendix: proofs of Lemma \ref{lem:Phitilda2} and Lemma \ref{lem:Phitilda1} }

We recall a result contained in \cite{KMW}. 
\begin{lem}
\label{lem:Musso}Suppose $n=5$ or $n\ge7$. We have 
\end{lem}
\begin{enumerate}
\item The function 
\[
\Phi_{0}:=\frac{1}{4(n-6)}\frac{1}{(|\bar{y}|^{2}+(1+y_{n})^{2})^{\frac{n-6}{2}}}+\frac{a_{1}}{(|\bar{y}|^{2}+(1+y_{n})^{2})^{\frac{n-2}{2}}}+a_{2},
\]
 for $a_{1},a_{2}\in\mathbb{R}$ satisfies
\begin{equation}
-\Delta\Phi_{0}:=\frac{1}{(|\bar{y}|^{2}+(1+y_{n})^{2})^{\frac{n-4}{2}}}\label{eq:Phi0}
\end{equation}
\item The function $\Phi_{1}:=\frac{1}{4(n-4)}\frac{y_{n}+1}{(|\bar{y}|^{2}+(1+y_{n})^{2})^{\frac{n-4}{2}}}+a_{1}\frac{y_{n}+1}{(|\bar{y}|^{2}+(1+y_{n})^{2})^{\frac{n}{2}}}=-\left(\frac{1}{n-4}\right)\partial_{n}\Phi_{0},$
for $a_{1}\in\mathbb{R}$ satisfies
\begin{equation}
-\Delta\Phi_{1}:=\frac{y_{n}+1}{(|\bar{y}|^{2}+(1+y_{n})^{2})^{\frac{n-2}{2}}}\label{eq:Phi1}
\end{equation}
\item The function $\Phi_{2}:=\frac{1}{2(n-4)}\frac{1}{(|\bar{y}|^{2}+(1+y_{n})^{2})^{\frac{n-4}{2}}}+\frac{a_{2}}{(|\bar{y}|^{2}+(1+y_{n})^{2})^{\frac{n-2}{2}}}+a_{2}',$
for $a_{2},a_{2}'\in\mathbb{R}$ satisfies
\begin{equation}
-\Delta\Phi_{2}:=\frac{1}{(|\bar{y}|^{2}+(1+y_{n})^{2})^{\frac{n-2}{2}}}\label{eq:Phi2}
\end{equation}
\end{enumerate}
\begin{proof}
The first claim is proved in \cite[Lemma A.1]{KMW} (in particular
in formula (A.2)). The second claim is proved again in \cite[Lemma A.1]{KMW},
while the last claim corresponds to \cite[Lemma A.2]{KMW}.
\end{proof}
\begin{lem}
\label{lem:Phitilda0}Let $n\ge5$. The function 
\[
\tilde{\Phi}_{0}=\left\{ \begin{array}{cc}
\frac{1}{6(n-8)}\frac{1}{(|\bar{y}|^{2}+(1+y_{n})^{2})^{\frac{n-8}{2}}}+a_{1}\frac{1}{(|\bar{y}|^{2}+(1+y_{n})^{2})^{\frac{n-2}{2}}}+a_{2} & \text{ for }n\neq8\\
-\frac{1}{12}\log(|\bar{y}|^{2}+(1+y_{n})^{2})+a_{1}\frac{1}{(|\bar{y}|^{2}+(1+y_{n})^{2})^{3}}+a_{2} & \text{ for }n=8
\end{array}\right.
\]
 for $a_{1},a_{2}\in\mathbb{R}$ satisfies
\begin{equation}
-\Delta\tilde{\Phi}_{0}:=\frac{1}{(|\bar{y}|^{2}+(1+y_{n})^{2})^{\frac{n-6}{2}}}.\label{eq:Phitilda0}
\end{equation}
\end{lem}
\begin{proof}
By change of variables we have that 
\begin{equation}
-\Delta\tilde{\Phi}_{0}(\bar{y},y_{n}-1)=\frac{1}{(|\bar{y}|^{2}+y_{n}^{2})^{\frac{n-6}{2}}}=\frac{1}{r^{n-6}},\label{eq:Phitilde-aux}
\end{equation}
where $r:=\sqrt{|\bar{y}|^{2}+y_{n}^{2}}$. So, in spherical coordinates,
set $\varphi_{0}(r)=\tilde{\Phi}_{0}(\bar{y},y_{n}-1),$ (\ref{eq:Phitilde-aux})
becomes 
\begin{equation}
-\varphi_{0}''-\frac{n-1}{r}\varphi_{0}'=\frac{1}{r^{n-6}}\label{eq:Phitilda-rad}
\end{equation}
and one can check that 
\[
\varphi_{0}(r)=\left\{ \begin{array}{cc}
\frac{1}{6(n-8)}\frac{1}{r^{n-8}}+\frac{a_{1}}{r^{n-2}}+a_{2} & \text{ for }n\neq8\\
-\frac{1}{6}\log r+\frac{a_{1}}{r^{6}}+a_{2} & \text{ for }n=8
\end{array}\right.
\]
 solves (\ref{eq:Phitilda-rad}).
\end{proof}
\begin{lem}
\label{lem:beta}Let $n=5$ or $n\ge7$. Set $\beta_{kl}:=\frac{\partial_{kl}^{2}\tilde{\Phi}_{0}}{(n-6)(n-4)}+\frac{\Phi_{0}}{(n-4)}\delta_{kl}$.
Then
\begin{equation}
-\Delta\beta_{kl}=y_{k}y_{l}U.\label{eq:beta}
\end{equation}
\end{lem}
\begin{proof}
By (\ref{eq:Phitilda0}) we have $-\Delta\partial_{kl}^{2}\tilde{\Phi}_{0}=\frac{(n-6)(n-4)y_{k}y_{l}}{(|\bar{y}|^{2}+(1+y_{n})^{2})^{\frac{n-2}{2}}}-\frac{(n-6)\delta_{kl}}{(|\bar{y}|^{2}+(1+y_{n})^{2})^{\frac{n-4}{2}}}$
and by (\ref{eq:Phi0}) we get the result. 
\end{proof}
Now we can achieve the prove of the two lemmas.
\begin{proof}[Proof of Lemma \ref{lem:Phitilda2}.]
Since $\bar{R}_{ijkk}=0$ have $\tilde{\Phi}_{2}:=\frac{1}{3}\bar{R}_{ijkl}(q)\partial_{ij}^{2}\left(\frac{\partial_{kl}^{2}\tilde{\Phi}_{0}}{(n-6)(n-4)}\right)=\frac{1}{3}\bar{R}_{ijkl}(q)\partial_{ij}^{2}\beta_{kl}$.
Thus, by (\ref{eq:beta}), we have 
\[
-\Delta\tilde{\Phi}_{2}=\frac{1}{3}\bar{R}_{ijkl}(q)\partial_{ij}^{2}(-\Delta\beta_{kl})=\frac{1}{3}\bar{R}_{ijkl}(q)\partial_{ij}^{2}\left(y_{k}y_{l}U\right)\frac{1}{3}\bar{R}_{ijkl}(q)y_{k}y_{l}\partial_{ij}^{2}U
\]
using the symmetry of the curvature tensor.
\end{proof}
\begin{proof}[Proof of Lemma \ref{lem:Phitilda1}.]
By Lemma \ref{lem:Phitilda2} and Lemma \ref{lem:Phitilda0} we have
that 
\[
-\Delta\left[\frac{\partial_{nn}^{2}\tilde{\Phi}_{0}}{(n-6)(n-4)}+\frac{\Phi_{0}}{n-4}+\Phi_{2}-2\Phi_{1}\right]=y_{n}^{2}U,
\]
so $\tilde{\Phi}_{1}=R_{ninj}(q)\partial_{ij}^{2}\left[\frac{\partial_{nn}^{2}\tilde{\Phi}_{0}}{(n-6)(n-4)}+\frac{\Phi_{0}}{n-4}+\Phi_{2}-2\Phi_{1}\right]$.
The claim follows by direct computation. 
\end{proof}


\begin{thebibliography}{10}
\bibitem{Al}S. Almaraz, \emph{A compactness theorem for scalar-flat
metrics on manifolds with boundary}, Calc. Var. \textbf{41} (2011)
341-386.

\bibitem{A2}S. Almaraz, \emph{Blow-up phenomena for scalar-flat metrics
on manifolds with boundary}, J. Differential Equations \textbf{251}
(2011), no. 7, 1813-1840.

\bibitem{A3}S. Almaraz, \emph{An existence theorem of conformal scalar-flat
metrics on manifolds with boundary}, Pacific J. Math. \textbf{248}
(2010), 1-22.

\bibitem{ALM}A. Ambrosetti, Y.Y. Li, A. Malchiodi, \emph{On the Yamabe
problem and the scalar curvature problems under boundary conditions}.
Math. Ann. \textbf{322} (2002), 667-699. 

\bibitem{Au1}T. Aubin, \emph{Equations differentielles non lineaires
et probleme de Yamabe concernant la courbure scalaire}, J. Math. Pures
Appl. \textbf{55} (1976), 269-296.

\bibitem{Au}T. Aubin, Some Nonlinear Problems in Riemannian Geometry.
Springer Monographs in Mathematics. Springer, Berlin (1998).

\bibitem{B}S. Brendle, \emph{Convergence of the Yamabe flow in dimension
6 and higher}, Invent. Math. \textbf{170} (2007), 541- 576. 

\bibitem{ch}S. S. Chen, \emph{Conformal deformation to scalar flat
metrics with constant mean curvature on the boundary in higher dimensions},
arxiv preprint https://arxiv.org/abs/0912.1302 (2010).

\bibitem{DK}M. Disconzi, M. Khuri, \emph{Compactness and non-compactness
for the Yamabe problem on manifolds with boundary}, J. Reine Angew.
Math. \textbf{724} (2017), 145-201. 

\bibitem{FA}V. Felli, M. Ould Ahmedou, \emph{Compactness results
in conformal deformations of Riemannian metrics on manifolds with
boundaries}, Math. Z. \textbf{244} (2003), 175-210.

\bibitem{GM}M.G. Ghimenti, A.M. Micheletti, \emph{A compactness result
for scalar-flat metrics on manifolds with umbilic boundary}, arXiv:1903.10990

\bibitem{GMP}M.G. Ghimenti, A.M. Micheletti, A. Pistoia, \emph{Blow-up
phenomena for linearly perturbed Yamabe problem on manifolds with
umbilic boundary, }J. Differential Equations \textbf{267} (2019),
587-618.

\bibitem{GMP1}M. Ghimenti, A.M. Micheletti, A. Pistoia, \emph{On
Yamabe type problems on Riemannian manifolds with boundary}, Pac.
J. Math. \textbf{284} (2016), 79-102. 

\bibitem{GMP2}M. Ghimenti, A.M. Micheletti, A. Pistoia, \emph{Linear
perturbation of the Yamabe problem on manifolds with boundary}, J.
Geom. Anal. \textbf{28} (2018), 1315-1340. 

\bibitem{Gi}G. Giraud, \emph{Sur la problème de Dirichlet généralisé}.
Ann. Sci. Ècole Norm. Sup. \textbf{46}, (1929) 131-145.

\bibitem{HL}Z.C. Han, Y. Li, \emph{The Yamabe problem on manifolds
with boundary: existence and compactness results}. Duke Math, J. \textbf{99}
(1999), 489-542.

\bibitem{HV} E. Hebey, M. Vaugon, \emph{Le probleme de Yamabe equivariant},
Bull. Sci. Math. \textbf{117} (1993), 241-286 

\bibitem{KMW}S. Kim, M. Musso, J. Wei, \emph{Compactness of scalar-flat
conformal metrics on low-dimensional manifolds with constant mean
curvature on boundary}, arXiv:1906.01317

\bibitem{KMS} M. Khuri, F. Marques, R. Schoen, \emph{A compactness
theorem for the Yamabe problem}. J. Differ. Geom. \textbf{81} (2009),
143-196 

\bibitem{Es}J. Escobar, \emph{Conformal deformation of a Riemannian
metric to a scalar flat metric with constant mean curvature on the
boundary}, Ann. Math. \textbf{136}, (1992), 1-50.

\bibitem{Es2}J. Escobar, \emph{Sharp constant in a Sobolev trace
inequality}, Indiana Univ. Math. J. \textbf{37}, (1988), 687-698.

\bibitem{LZ} Y.Y. Li, M. Zhu, \emph{Yamabe type equations on three
dimensional Riemannian manifolds}. Commun. Contemp. Math. \textbf{1}
(1999), 1-50.

\bibitem{MN}M. Mayer, C.B. Ndiaye, \emph{Barycenter technique and
the Riemann mapping problem of Cherrier-Escobar}. J. Differential
Geom. \textbf{107} (2017), no. 3, 519-560. 

\bibitem{M1}F. Marques, \emph{Existence results for the Yamabe problem
on manifolds with boundary}, Indiana Univ. Math. J. \textbf{54} (2005)
1599-1620. 

\bibitem{M3}F. Marques, \emph{A priori estimates for the Yamabe problem
in the non-locally conformally flat case}, J. Differ. Geom. \textbf{71}
(2005) 315-346. 

\bibitem{M2}F. Marques, \emph{Compactness and non compactness for
Yamabe-type problems}, Progress in Nonlinear Differential Equation
and Their Applications, \textbf{86} (2017) 121-131.

\bibitem{SZ}R. Schoen, D. Zhang, \emph{Prescribed scalar curvature
on the n-sphere}, Calc. Var. Partial Differ. Equ. \textbf{4} (1996),
1-25.

\bibitem{S}R. Schoen, \emph{Conformal deformation of a Riemannian
metric to constant scalar curvature}, J. Differ. Geom. \textbf{20}
(1984), 479-495.

\bibitem{T}N. Trudinger, \emph{Remarks concerning the conformal deformation
of Riemannian structures on compact manifolds}, Annali Scuola Norm.
Sup. Pisa \textbf{22} (1968), 265-274.

\bibitem{Y}H. Yamabe, \emph{On a deformation of Riemannian structures
on compact manifolds}, Osaka Math. J. \textbf{12} (1960), 21-37.
\end{thebibliography}
\end{document}